\documentclass[11pt]{amsart}
\usepackage{amssymb, amsfonts, amsmath, stackrel, tikz, tikz-cd, pgfplots, verbatim, xypic, enumerate, todonotes}
\usetikzlibrary{intersections, calc, matrix, decorations.pathreplacing}
\usepackage[outline]{contour} \contourlength{2pt}
\usepackage{hyperref}
\hypersetup{
    colorlinks=false, 
    linktoc=all,     
    linkcolor=blue,  
}

\pgfplotsset{compat=newest}

\usepackage{geometry}
\allowdisplaybreaks

\theoremstyle{definition}
\newtheorem{theorem}{Theorem}[section]
\newtheorem{lemma}[theorem]{Lemma}
\newtheorem{corollary}[theorem]{Corollary}
\newtheorem{proposition}[theorem]{Proposition}
\newtheorem{definition}[theorem]{Definition}

\newtheorem{example}[theorem]{Example}
\newtheorem{notation}[theorem]{Notation}
\newtheorem{remark}[theorem]{Remark}
\numberwithin{equation}{section}
\numberwithin{figure}{section}

\numberwithin{equation}{section}
\newcommand{\C}{\mathbb{C}}

\newcommand{\Q}{\mathbb{Q}}
\newcommand{\R}{\mathbb{R}}
\newcommand{\Z}{\mathbb{Z}}

\newcommand{\al}{\alpha}

\newcommand{\mc}{\mathcal }
\newcommand{\tr}{\mathrm{tr}}
\newcommand{\bu}{\bullet}

\newcommand{\om}{\omega}
\newcommand{\Del}{{\bf \Delta}}

\newcommand{\Cech}{{\v{C}ech} }  
\newcommand{\Tot}{{\bf Tot}}
\newcommand{\tot}{{\bf tot}}
\newcommand{\Ch}{{\bf Ch}}
\newcommand{\ChL}{{\bf Ch^{\scalebox{0.5}{$\searrow$}}}}
\newcommand{\ChLhol}{{{\bf Ch}_{hol}^{\scalebox{0.5}{$\searrow$}}}}
\newcommand{\BCh}{{{\bf Ch}_{\Delta}}}

\newcommand{\ch}{{\textrm{Ch}}}
  \newcommand{\ul}{[u]^{\bu\leq 0}}

\newcommand{\sAb}{\mc{A}b^{\Del^{op}}}
\newcommand{\sSet}{\mc{S}et^{\Del^{op}}}
\newcommand{\Chain}{\mc {C}h}
\newcommand{\Set}{\mc{S}et}

\newcommand{\OM}{{\bf \Omega}}

\newcommand{\CMan}{\mathbb{C}\mc{M}an}
\newcommand{\Man}{\mc{M}an}
\newcommand{\Ob}{{\bf Obj}}
\newcommand{\Mor}{{\bf Mor}}

\newcommand{\sgn}{\mathrm{sgn}}

\newcommand{\Omhol}[1]{\Om_{0,hol}^{#1}} \newcommand{\Ohol}{\Om_{0,hol}}
\newcommand{\OMhol}{{\bf \Omega}}
\newcommand{\Omdelhol}[1]{\Om_{\partial,hol}^{#1}} \newcommand{\Odelhol}{\Om_{\partial,hol}}
\newcommand{\OMdelhol}{{\bf \Omega_{\textrm{\bf $\partial$,hol}}}}
\newcommand{\Nerve}{{\mathcal N}}

\newcommand{\DK}{DK}
\newcommand{\DKSet}{\underline{DK}}
\newcommand{\Om}{\Omega}

\newcommand{\quot}{Q}

\newcommand{\CN}{\mathrm{\check{N}}}   \newcommand{\CechNerve}{\mathrm{\check{N}}}
\newcommand{\NU}{\mathrm{\check{N}}\mc U}
   
\newcommand{\vC}{{\check{C}}}   
\newcommand{\Delt}[1]{\Delta^{#1}}

\newcommand{\bbb}{\mathfrak{b}}  

\newcommand{\vb}{{\mathcal{VB}}}

\newcommand{\vbn}{{\mathcal{VB}^\nabla}}
\newcommand{\fvbn}{{\mathcal{FVB}^\nabla}}
\newcommand{\pvbn}{{\mathcal{PVB}^\nabla}}
\newcommand{\pbn}{{\mathcal{PVB}^d}}   

\newcommand{\HVBnabla}{{\mathcal{HVB}^\nabla}}
\newcommand{\VB}{{\bf{VB}}^{\nabla}}
\newcommand{\BVB}{{\bf{VB}}_{\Delta}}
\newcommand{\FVB}{{\bf{FVB}}^{\nabla}}
\newcommand{\PVB}{{\bf{PVB}}^{\nabla}}
\newcommand{\BUN}{{\bf{VB}}}

\newcommand{\HVB}{{{\bf HVB}^\nabla}}

\newcommand{\Top}{\mc{T}op}
\newcommand{\OMdR}{{\bf \Omega_{\textrm{\bf dR}}}}
 \newcommand{\OdR}{\Om_{\textrm{dR}}}
\newcommand{\Chu}{\textrm{Ch}_u}
\newcommand{\nab}[1]{\nabla^{\{#1\}}}

\newcommand{\nabladelta}{{\nabla_{\hspace{-1mm}\Delta}}}
\newcommand{\delbar}{\bar{\partial}}
\newcommand{\del}{\partial}

\newcommand{\CechSh}[1]{{#1}^{\check{\dagger}}} 
\newcommand\restr[2]{{
  \left.\kern-\nulldelimiterspace 
  #1 
  \vphantom{\big|} 
  \right|_{#2} 
  }}

\DeclareMathOperator*{\colim}{colim}

\author[C. Glass]{Cheyne Glass}
\address{Cheyne Glass, State University of New York at New Paltz, Department of Mathematics, 1 Hawk Dr., New Paltz, NY 12561}
  \email{glassc@newpaltz.edu}

\author[T. Tradler]{Thomas Tradler}
  \address{Thomas Tradler, New York City College of Technology The City University of New York, Department of Mathematics, 300 Jay Street, Brooklyn, NY 11201}
  \email{ttradler@citytech.cuny.edu}

\author[M. Zeinalian]{Mahmoud Zeinalian}
  \address{Mahmoud Zeinalian, Lehman College, The City University of New York, Department of Mathematics, 250 Bedford Park Blvd W, Bronx, NY 10468}
  \email{mahmoud.zeinalian@lehman.cuny.edu}

\title{Cech - de Rham Chern character on the stack of holomorphic vector bundles}



\begin{document}
\maketitle
\begin{abstract}
We provide a formula for the Chern character of a holomorphic vector bundle in the hyper-cohomology of the de Rham complex of holomorphic sheaves on a complex manifold. This Chern character can be thought of as a completion of the Chern character in Hodge cohomology obtained as the trace of the exponential of the Atiyah class, which is \v{C}ech closed, to one that is \v{C}ech-Del closed. Such a completion is a key step toward lifting O'Brian-Toledo-Tong invariants of coherent sheaves from Hodge cohomology to de Rham cohomology. An alternate approach toward the same end goal, instead using simplicial differential forms and Green complexes, can be found in Hosgood's works \cite{Ho1,Ho2}. In the algebraic setting, and more generally for K\"{a}hler manifolds, where Hodge and de Rham cohomologies agree, such extensions are not necessary, whereas in the non-K\"{a}hler, or equivariant settings the two theories differ. We provide our formulae as a map of simplicial presheaves, which readily extend the results to the equivariant setting and beyond. This paper can be viewed as a sequel to  \cite{GMTZ1} which covered such a discussion in Hodge cohomology. As an aside, we give a conceptual understanding of how formulas obtained by Bott and Tu for Chern classes using transition functions and those from Chern-Weil theory using connections, are part of a natural unifying story.
\end{abstract}
\setcounter{tocdepth}{1}
\tableofcontents

\section{Introduction}\label{SEC:Intro}

Complex vector bundles have characteristic classes which are cohomology classes in the base of the bundle. In the smooth setting, one can give cocycle level de Rham representatives of these classes by fixing a connection and applying various invariant polynomials to the curvature. That the cohomology class of this representative is well-defined is witnessed by the Chern-Simons forms for a path of connections. In fact, there is a hierarchy of Chern-Simons forms interpolating choices in each step. One natural question is whether these invariants can be constructed combinatorially using the data of local trivializations and transition functions. Such formulae have been given by Atiyah in the holomorphic case through what we now call the Atiyah class. Atiyah's procedure extracts forms that naturally land in the Hodge cohomology of the base. For a K\"ahler base, the Hodge cohomology agrees with the de Rham cohomology and these invariants agree with the de Rham invariants. For a general complex manifold, the Hodge and de Rham cohomologies are related through a spectral sequence and the usual Chern character, as the trace of the exponential of the Atiyah class, lies in the first page of this spectral sequence identified as Hodge cohomology. It is therefore natural to ask if this Chern character is a shadow of a naturally assigned class in de Rham cohomology identified as the infinity-page of the above spectral sequence. In this paper, we provide explicit formulae for such a class in terms of the transition functions of holomorphic vector bundles. Our motivation for this work was to complete O'Brian-Toledo-Tong's Chern character for coherent sheaves, which landed in Hodge cohomology, to a class in de Rham cohomology. While our attempt is still a work in progress, it has become clear that even in the case of holomorphic vector bundles presented in this paper the task is novel and sufficiently complicated. Furthermore, our general framework of simplicial presheaves naturally extends the above discussions to equivariant (Hodge and de Rham) settings and beyond.  

In slightly more detail, for a smooth manifold $M$, the Chern character associates to a complex vector bundle $\mc E$ over $M$ a cohomology class $[\ch(\mc E)]\in H^\bu(M)$. There are also similar higher Chern-Simons classes that can be associated to isomorphisms of complex vector bundles, and more generally to elements in the nerve $\Nerve(\vb(M))$ of the category $\vb(M)$ of complex vector bundles over $M$ whose morphisms consist of isomorphisms of complex vector bundles. To give an explicit representative of these classes, one needs to make certain choices, such as, for example following Chern-Weil theory, the choice of a connection $\nabla$ on $\mc E$, with which we can obtain a de Rham representative of the Chern character via $\ch(\mc E)=\tr\big(e^{\frac{R}{2\pi i}}\big)\in \OdR^\bu(M)$, where $R$ denotes the curvature of $\nabla$. If we furthermore want to encode higher Chern-Simons-type information, we may consider the category $\vbn(M)$ of complex vector bundles with connections whose morphisms consist of isomorphisms that do not need to respect the connection in any way. Taking the nerve of this category, denoted by $\VB(M)=\Nerve(\vbn(M))$, we can indeed find a Chern character map $\Ch(M):\VB(M)\to \OMdR(M)$, which encodes the Chern character and higher Chern-Simons type forms via explicit representatives in a natural way, (i.e., as a natural transformation of simplicial presheaves). Here, $\OMdR(M)$ is essentially the de Rham forms on $M$ arranged in a convenient way (see definition \ref{DEF:OMdR}).

We describe this map by considering a somewhat modified simplicial presheaf $\BVB(M)$, given by considering complex vector bundles with connections over $M\times |\Delta^n|$, where $|\Delta^n|$ denotes the standard $n$-simplex (see definition \ref{DEF:VB_Delta}). Using this as our basic example, we define a straightforward Chern character map $\BCh(M):\BVB(M)\to \OMdR(M)$ given by taking the trace of the exponentiated curvature (as usual) and integrating out the fiber $|\Delta^n|$ (see definition \ref{DEF:basic-CS-map}, \eqref{EQU:CH(M)_n(E)(e)}, and proposition \ref{PROP:Ch-presheaf-map}). This map acts as a guide for all Chern character constructions in this paper. Indeed, the Chern character map $\Ch(M):\VB(M)\to \OMdR(M)$ mentioned above is given by simply taking an element of the nerve $\VB(M)=\Nerve(\vbn(M))$, interpreting it as an element of $\BVB(M)$ (as described in definition \ref{DEF:E-to-EDelta}), and applying $\BCh(M):\BVB(M)\to \OMdR(M)$. It is worth noting, that the assignment going from $\VB(M)$ to $\BVB(M)$ via definition \ref{DEF:E-to-EDelta} does not produce a map simplicial presheaves, as it does not respect the simplicial structures, but it does become a map of simplicial presheaves after composing it with $\BCh(M)$ (due to the application of the trace, which removes any discrepancies in the assignment from definition \ref{DEF:E-to-EDelta}).
\begin{equation*}
\xymatrix{ \VB(M) \ar@{.>}[r] \ar_{\Ch(M)}@/_1pc/[rr]
& \BVB(M) \ar^{\BCh(M)}[r]  
& \OMdR(M) }
\end{equation*}

Using the fact the we have a map of simplicial presheaves, we can give yet another way to encode the Chern character, other than making a choice of connection as in Chern-Weil theory, but rather by a choice of local trivializations and transitions functions. More precisely, we can use the fact that every complex vector bundle $\mc E\to M$ can be decomposed into locally trivializable constituents $\mc E_i \xrightarrow{s_i} U_i\times \C^n\to U_i$ where $\{ U_i\}_i$ is a suitable open cover of $M$. Thus, it is natural to consider the category $\pvbn(M)$ of trivial complex vector bundles  $\C^n \times M \to M$ with \emph{flat} connections $d_{DR}$, and morphisms given by isomorphisms in $\Gamma(M, GL_n(\mathbb{C}))$. The nerve of this category is denoted by $\PVB(M)=\Nerve(\pvbn(M))$. Since there is a natural inclusion $\PVB(M)\hookrightarrow \VB(M)$ of trivial vector bundles into all vector bundles over $M$, we also obtain an induced Chern character map, which, by slight abuse of notation, is also denoted by
\[
\Ch(M):\PVB(M)\hookrightarrow \VB(M)\stackrel{\Ch(M)}\longrightarrow \OMdR(M).
\]
The invariants of $\mc E$ can then be recovered by taking the totalization of the Chern character on a chosen cover $\{U_i\}_i$.

In section \ref{SEC:CW-vs-BT} we show how the two approaches for the Chern character via Chern-Weil theory or via a formula given by Bott and Tu are coming from specific choices of global or local connections on our complex vector bundle (see examples \ref{EXA:Chern-Weil-example} and \ref{EXA:Bott-Tu-example}). Moreover, from our general setup, we also obtain an immediate identification of the two in the appropriate cohomology theory (see corollary \ref{COR:CW-id-BT}).

In section \ref{SEC:holomorphic-Chern-extended}, we study the analogous situation for holomorphic vector bundles. This compares to a similar Chern character map that we studied in \cite{GMTZ1}, where we gave a map of simplicial presheaves encoding the trace of the exponentiated Atiyah class. A major difference in our current setting is that the output of our construction lands in a space which explicitly is using the $\del$-operator of holomorphic forms, whereas in \cite{GMTZ1} we explicitly had taken a zero differential (see definition \ref{DEF:Omega-hol}). This has the effect that the Chern character map from \cite{GMTZ1} is just a first level approximation of the Chern character given here; or, in other words, ``turning on the $\del$'' requires us to provide higher terms that will correct for the fact that some terms in the Chern character map from \cite{GMTZ1} are not $\del$-closed. These features were also observed by Hosgood \cite{Ho1,Ho2} and we refer the reader there for some explicit derivations of these higher terms. The construction given in section \ref{SEC:basic-example} via the basic example almost provides such an extension, with the minor caveat that the construction from section \ref{SEC:basic-example} only gives a symmetrized version of the map we want (see example \ref{EXA:lowest-Chern-symmetrized}). However, it is possible to make a slight (but technical) modification to the construction from section \ref{SEC:basic-example} and obtain a Chern character map $\ChL: \VB \to \OMdR$ with the expected lowest terms reminiscent to \cite{GMTZ1} (see proposition \ref{EQU:CHL-lowest-ul}):
\begin{equation*}
\ChL(M)_n(\mc E)(e_{i_0,\dots, i_\ell})=\frac{u^\ell\cdot \sgn(\ell,\dots,1)}{\ell!(2\pi i)^\ell} \cdot \tr\big(\tilde{\theta}_{\ell}\dots \tilde{\theta}_{1}\big)+\sum_{p\geq \ell+1}u^{p}\cdot (\dots)
\end{equation*}
This map induces a Chern character on holomorphic vector bundles $\ChLhol:\HVB\to \OMdelhol$  (see proposition \ref{PROP:Ch->-factors-to-hol}), providing a completion of the Hodge Chern character map from \cite{GMTZ1} to a de Rham Chern character. We note, that after totalizing, this is akin to getting an output for a holomorphic vector bundle in the cohomology $H^\bu(\vC^\bu(\mc U,\Omhol \bu),\delta)$ which is Hodge cohomology, versus having it in $H^\bu(\vC^\bu(\mc U,\Odelhol^\bu),\delta+\del)$ which is de Rham cohomology (see remark \ref{REM:smooth-del-or-no-del}).

\section{A basic example}\label{SEC:basic-example}

In this section we define a basic example of a simplicial presheaf and its Chern character map, which will be a guide for us in the following sections. The main idea here, similar to Hosgood's approach \cite{Ho1,Ho2}, is to adapt the machinery of Chern-Weil theory to a greater generality. We begin by reviewing some notation.

\begin{notation}\label{NOTATION:simplicial-sets}
Denote by $\Set$ the category of (possibly large) sets, by $\Man$ the category of closed, smooth manifolds, by $\Del$ the simplicial category (see  \cite[Appendix A]{GMTZ1}), and denote their opposite categories by $\Set^{op}$, $\Man^{op}$, and $\Del^{op}$, respectively. A simplicial set is a functor $\Del^{op}\to \Set$, and we denote the category of all simplicial sets by $\sSet$. For example, $\Delta^n:\Del^{op}\to \Set$, $(\Delta^n)_k:=\Del([k],[n])$ is the simplicial set of the standard $n$-simplex. We refer the reader to \cite[Appendix A]{GMTZ1} for some of the notation.

Denote by $|\Delta^n|$ the geometric realization of $\Delta^n$, which is a topological space. $|\Delta^n|$ can be identified with the standard topological $n$-simplex
\begin{equation}\label{EQU:standard-n-simplex}
|\Delta^n| = \{(t_1,\dots,t_n)\in \R^n\,|\,0\leq t_1\leq \dots \leq t_n\leq 1\}.
\end{equation}
The $|\Delta^n|$s may be assembled to a cosimplicial topological space, i.e., to a functor $|\Delta^\bu|:\Del\to \Top$, where $\Top$ denotes the category of topological spaces, by mapping $[n]\mapsto |\Delta^n|$. The face maps are given by $\delta_j:|\Delta^{n-1}|\to |\Delta^{n}|$, $\delta_j(t_1,\dots,t_j, \dots, t_{n-1})=(t_1,\dots,t_j, t_j, \dots, t_{n-1})$ for $j=1,\dots, n-1$, while $\delta_0(t_1, \dots, t_{n-1})=(0,t_1,\dots, t_{n-1})$ and $\delta_n(t_1, \dots, t_{n-1})=(t_1,\dots, t_{n-1},1)$. The degeneracy maps are given by $\sigma_j:|\Delta^{n}|\to |\Delta^{n-1}|$, $\sigma_j(t_1,\dots, t_n)=(t_1,\dots,\widehat{t_{j+1}},\dots t_n)$ for $j=0,\dots, n-1$. 

If $\{i_0,\dots, i_\ell\}\sqcup \{j_1,\dots, j_{n-\ell}\}= \{0,\dots, n\}$ is a partition of $\{0,\dots, n\}$ with $0\leq i_0< \dots < i_\ell\leq n$ and $0\leq j_1< \dots< j_{n-\ell}\leq n$, then we denote by 
\begin{equation}\label{DEF:|e_i0...in|}
|e_{i_0,\dots, i_\ell}|=\delta_{j_{n-\ell}}\circ\dots\circ \delta_{j_1}(|\Delta^\ell|)\subseteq |\Delta^n|
\end{equation}
the $\ell$-subcell of $|\Delta^n|$.

In particular, the $i$th vertex $|e_i|\subseteq |\Delta^n|$, where $0\leq i\leq n$, is the point with the coordinates
\begin{equation}\label{EQU:ith-vertex}
 t_1=\dots=t_i=0\quad\text{ and }\quad t_{i+1}=\dots= t_n=1.
\end{equation}
\end{notation}

\begin{definition}\label{DEF:VB_Delta}
We define a functor $\BVB:\Man^{op}\to \sSet$ as follows. For a smooth manifold $M\in \Ob(\Man)$, define $\BVB(M)$ to be the simplicial set, whose $n$-simplicies $\BVB(M)_n$ consist of the (large) set of complex vector bundles with connection over $M\times |\Delta^n|$,
\begin{multline}\label{EQU:BVB(M)-definition}
\BVB(M)_n
:=\{\mathcal E=(E,\nabla)\,\,|
\\
\,\,E\to M\times |\Delta^n| \text{ is a complex vector bundle with connection }\nabla\text{ on }E\}.
\end{multline}
Morphisms $\rho:[n]\to [m]$ of $\Del$ induce maps $\BVB(M)_\rho:\BVB(M)_m\to \BVB(M)_n$ via pullbacks. More precisely, denote the map induced on $|\Delta^\bu|$ from $\rho:[n]\to [m]$ also by $\rho:|\Delta^{n}|\to |\Delta^{m}|$. This then yields the map $\BVB(M)_\rho:\BVB(M)_m\to \BVB(M)_{n}$ via  $\BVB(M)_\rho(E\to M\times |\Delta^m|,\nabla)=((id_M\times \rho)^*(E)\to M\times |\Delta^{n}|,(id_M\times \rho)^*(\nabla))$.\footnote{\label{FN:lax}While the composition of pullbacks does not agree set theoretically with the pullback of the compositions, it does so up to canonical isomorphism. Therefore, the structures defined above are more precisely ``lax'' structures, such as lax simplicial sets, lax presheaves, etc.}

Then, $\BVB:\Man^{op}\to \sSet$ becomes a functor\footnotemark[\value{footnote}] via pullback, i.e., for a smooth map $g:M'\to M$, define $\BVB(M)_n\to \BVB(M')_n$ via the pullback $(E\to M\times |\Delta^n|,\nabla)\mapsto ((g\times id_{|\Delta^n|})^*(E)\to M'\times |\Delta^n|,(g\times id_{|\Delta^n|})^*(\nabla))$, which assembles to a map of simplicial sets  $\BVB(M)\to \BVB(M')$.
\end{definition}

Our Chern character map will have $\BVB(M)$ as its domain, and we next define its range $\OMdR$. The definition of $\OMdR$ closely follows the corresponding definition in \cite[Definition 2.2]{GMTZ1}, but with the crucial difference that now we will use a non-vanishing de Rham differential $d$.

\begin{definition}\label{DEF:OMdR}
We define the de Rham functor as a functor $\OMdR:\Man^{op}\to \sSet$ to be the composition of functors $\OMdR:=\mc F\circ \DK\circ\quot\circ \OdR^\bu(-)$ defined below:
\[
\OMdR:\Man^{op}
\stackrel{\OdR^\bu(-)}{\longrightarrow}\Chain^+
\stackrel{\quot}{\longrightarrow}\Chain^-
\stackrel{\DK}{\longrightarrow}\sAb
\stackrel{\mc F}{\longrightarrow}\sSet.
\]
Here, the four functors $\OdR^\bu(-)$, $\quot$, $\DK$, $\mc F$ are defined as follows:
\begin{itemize}
\item
The first functor $\OdR^\bu(-):\Man^{op}\to\Chain^+$ is the de Rham functor, i.e., for a smooth manifold $M\in \Ob(\Man)$, $\OdR^\bu(M)$ is the (non-negatively graded) cochain complex of \emph{complex valued} de Rham forms $\OdR^\bu(M)=\OdR^\bu(M,\C)$ on $M$ with the de Rham differential $d=d_{\textrm{dR}}$. \item
Next, the functor $Q:\Chain^+\to \Chain^-$ maps a non-negatively graded cochain complex $C=C^{\bu\geq 0}$ to the non-positively graded cochain complex $$Q(C):=C[u]^{\bu\leq 0}=\frac{C[u]}{(C[u])^{\bu >0}},$$ where $u$ is a formal variable of degree $\|u\|=-2$, cf. \cite[Definition B.5]{GMTZ1}.
\item
The Dold-Kan functor $\DK:\Chain^-\to \sAb$ is given as the simplicial abelian group, which in simplicial degree $n$ consists of the cochain maps from normalized cells of the standard $n$-simplex $\Delt{n}$ to $C^{\bu\leq 0}$:
$$\DK(C^{\bu\leq 0})_n:={\Chain^-}(N(\Z\Delt{n}),C^{\bu\leq 0})$$
\item
Finally, $\mathcal F:\sAb\to \sSet$ is the forgetful functor from simplicial abelian groups to simplicial sets. 
\end{itemize}
In particular, the composition $\DKSet :=\mathcal F\circ \DK: \Chain^-\to \sSet$ is the underlying simplicial set of the Dold-Kan functor, cf. \cite[Definition B.3]{GMTZ1}. Since $N(\Z\Delt{n})$ is generated by non-degenerate simplicies of the standard $n$-simplex, we informally think of an $n$-simplex in $\DKSet(C^{\bu\leq 0})_n$ as a labeling of the standard simplex with elements in $C^{\bu\leq 0}$ with the appropriate face conditions. For example, a $2$-simplex in $\DKSet(C^{\bu\leq 0})_2$ is given by the following data:
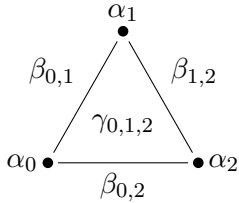
\[
\begin{tikzpicture}[scale=0.5]
\node (E0) at (0, -0.5) {}; \fill (E0) circle (4pt) node[left] {$\alpha_0$};
\node (E1) at (2, 3) {}; \fill (E1) circle (4pt) node[above] {$\alpha_1$};
\node (E2) at (4, -0.5) {}; \fill (E2) circle (4pt) node[right] {$\alpha_2$};
\draw [-] (E0) -- node[above left] {$\beta_{0,1}$} (E1);
\draw [-] (E1) -- node[above right] {$\beta_{1,2}$} (E2);
\draw [-] (E0) -- node[below] {$\beta_{0,2}$} (E2);
\node (C) at (2,0.5) {$ \gamma_{0,1,2}$};
\node[right] at (7,3.4) {$\alpha_1,\alpha_2,\alpha_3\in C^0 \quad,\quad \beta_{0,1},\beta_{1,2},\beta_{0,2}\in C^{-1} \quad,\quad \gamma_{1,2,3}\in C^{-2}$};
\node[right] at (7.9,2.4) {such that:};
\node[right] at (12,2.2) {$d_C(\gamma_{0,1,2})=\beta_{1,2}-\beta_{0,2}+\beta_{0,1}$};
\node[right] at (12,1.2) {$d_C(\beta_{0,1})=\alpha_1-\alpha_0$};
\node[right] at (12,0.2) {$d_C(\beta_{1,2})=\alpha_2-\alpha_1$};
\node[right] at (12,-0.8) {$d_C(\beta_{0,2})=\alpha_2-\alpha_0$};
\end{tikzpicture}
\]

Thus, we defined the simplicial set
\[
\OMdR(M)=\DKSet(\OdR^\bu(M)\ul)={\Chain^-}(N(\Z\Delt{\bu}),\OdR^\bu(M)\ul)
\]
\end{definition}

We next define a map $\BCh: \BVB \to \OMdR$ of simplicial presheaves. To this end, recall that for a vector bundle with connection, $\mc E=(E,\nabla)$, the Chern character is the even form 
\begin{equation}
\textrm{Ch}(\mc E)=\tr\Big(\exp\big(\frac{R}{2\pi i}\big)\Big)=\sum_{k\geq 0} \frac{\tr(R^k)}{k!\cdot (2\pi i)^k}\quad \in \OdR^{even}(M),
\end{equation}
where $R=\nabla^2\in \Om^2(M,End(E))$ is the curvature of the connection $\nabla$ on $E$, and $i=\sqrt{-1}$. Note that the $\textrm{Ch}(\mc E)$ is closed under the de Rham differential, due to the Bianchi identity $\tilde\nabla R=0$ (where $\tilde \nabla$ is the induced connection on $End(E)$) used in $d(\tr(R^k))=\tr(\tilde \nabla(R^k))=\tr(\sum_j R\dots R(\tilde\nabla R) R\dots R)=0$. It is well-known that the cohomology class of the Chern character is independent of the connection and only depends on the underlying vector bundle. For our purposes, it will be useful to rewrite the Chern character as an element in total degree $0$ by introducing a formal variable $u$ of degree $\|u\|=-2$, so that the degree $\|u\cdot R\|=0$. We thus define the following polynomial in $u$:
\begin{equation}\label{EQU:ChuE-expanded}
\Chu(\mc E):=\tr\Big(\exp\big(\frac{u\cdot R}{2\pi i}\big)\Big)=\sum_{k\geq 0} \frac{u^k\cdot \tr(R^k)}{k!\cdot (2\pi i)^k}\quad \in \OdR^{\bullet}(M)[u],
\end{equation}
which is of total degree $\|\Chu(\mc E)\|=0$. By the same argument as above, $\Chu(\mc E)$ is de Rham $d$-closed.

\begin{definition}\label{DEF:basic-CS-map}
To define the Chern character map $\BCh: \BVB \to \OMdR$ we have to give a map of simplicial sets $\BCh(M):\BVB(M) \to \OMdR(M)=\DKSet(\OdR^\bu(M)\ul)$, i.e., a map $\BCh(M)_n:\BVB(M)_n \to \DKSet(\OdR^\bu(M)\ul)_n={\Chain^-}(N(\Z\Delt{n}),\OdR^\bu(M)\ul)$ for each manifold $M\in \Ob(\Man)$ and each simplicial degree $n=0, 1, 2, \dots$. Thus, for a vector bundle $\mc E=(E,\nabla)$ over $M\times |\Delta^n|$, we need to give a morphism $\BCh(M)_n(\mc E)\in {\Chain^-}(N(\Z\Delt{n}),\OdR^\bu(M)\ul)$. Now, $N(\Z\Delt{n})$ is the free abelian group with generators $e_{i_0,\dots, i_\ell}$ in degree $(-\ell)$, for any $0\leq i_0< \dots < i_\ell\leq n$ and $\ell=0,\dots, n$ (see \cite[Example B.2]{GMTZ1}), and we recall that we denote by $|e_{i_0,\dots, i_\ell}|\subseteq |\Delta^n|$ the corresponding $\ell$-subcell of $|\Delta^n|$. Then, $\BCh(M)_n(\mc E)$ is mapping the $\ell$-cell $e_{i_0,\dots,e_\ell}\in N(\Z\Delt{n})$ to $\BCh(M)_n(\mc E)(e_{i_0,\dots, i_\ell})\in (\OdR^\bu(M)[u])^{-\ell}$ via the assignment
\begin{align}\label{EQU:CH(M)_n(E)(e)}
\BCh(M)_n(\mc E)(e_{i_0,\dots, i_\ell})
&:= \int_{|e_{i_0,\dots, i_\ell}|}\Chu\Big(\mc E\big|_{M\times |e_{i_0,\dots, i_\ell}|}\Big) \\
&\nonumber =\sum_{k\geq 0} \frac{u^k}{k!\cdot (2\pi i)^k}\int_{|e_{i_0,\dots, i_\ell}|}\tr\Big((R|_{M\times |e_{i_0,\dots, i_\ell}|})^k\Big).
\end{align}
In words, we restrict $\mc E$ to the given $\ell$-cell $M\times |e_{i_0,\dots, i_\ell}|$ and integrate over the fiber $|e_{i_0,\dots, i_\ell}|$ of $M\times |e_{i_0,\dots, i_\ell}|$, which yields an element of total degree $\|\BCh(M)_n(\mc E)(e_{i_0,\dots, i_\ell})\|=-\ell$. Informally, we think of $\BCh(M)_n(\mc E)$ as labeling the cells of the standard $n$-simplex with elements in $\OdR^\bu(M)\ul$:
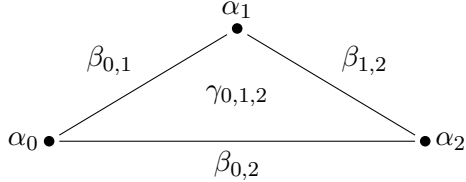
\begin{equation*}
\begin{tikzpicture}[scale=0.5]
\node (E0) at (0, 0) {}; \fill (E0) circle (4pt) node[left] {$\alpha_0$};
\node (E1) at (5, 3) {}; \fill (E1) circle (4pt) node[above] {$\alpha_1$};
\node (E2) at (10, 0) {}; \fill (E2) circle (4pt) node[right] {$\alpha_2$};
\draw [-] (E0) -- node[above left] {$\beta_{0,1}$} (E1);
\draw [-] (E1) -- node[above right] {$\beta_{1,2}$} (E2);
\draw [-] (E0) -- node[below] {$\beta_{0,2}$} (E2);
\node (C) at (5,1.2) {$ \gamma_{0,1,2}$};
\end{tikzpicture}
\end{equation*}
where for $j\in \{0,1,2\}$ and $(j',j'')\in \{(0,1),(0,2),(1,2)\}$:
\begin{align*}
\alpha_j=& \tr(Id)+u\cdot \frac{ \tr(R|_{M\times |e_j|})}{2\pi i}+u^2\cdot\frac{ \tr((R|_{M\times |e_j|})^2)}{2!\cdot (2\pi i)^2}+\dots, \\
\beta_{j',j''}=& \int_{|e_{j',j''}|}\tr(Id)+u\cdot \int_{|e_{j',j''}|}\frac{ \tr(R|_{M\times |e_{j',j''}|})}{2\pi i}+u^2\cdot\int_{|e_{j',j''}|}\frac{ \tr((R|_{M\times |e_{j',j''}|})^2)}{2!\cdot (2\pi i)^2}+\dots, \\
\gamma_{0,1,2}=& \int_{|e_{0,1,2}|}\tr(Id)+u\cdot \int_{|e_{0,1,2}|}\frac{ \tr(R|_{M\times |e_{0,1,2}|})}{2\pi i}+u^2\cdot\int_{|e_{0,1,2}|}\frac{ \tr((R|_{M\times |e_{0,1,2}|})^2)}{2!\cdot (2\pi i)^2}+\dots.
\end{align*}
Note that in the above series the lower $u$-terms vanish whenever the integral is over a de Rham form of lower degree; more precisely, $\BCh(M)_n(\mc E)(e_{i_0,\dots, i_\ell})$ will vanish in powers of $u$ for $u^0, \dots, u^{\lfloor\frac{\ell-1} 2\rfloor}$. Proposition \ref{PROP:Ch-presheaf-map} below states that $\BCh$ gives indeed a map of simplicial presheaves.
\end{definition}

\begin{proposition} \label{PROP:Ch-presheaf-map}
$\BCh: \BVB \to \OMdR$ is a map of simplicial presheaves, i.e., a natural transformation between $\BVB$ and $\OMdR$.
\end{proposition}
\begin{proof}
We need to check three properties, first that $\BCh(M)_n(\mc E)$ is a chain map, second that $\BCh(M)$ is a map of simplicial sets, and third that $\BCh$ is a natural transformation. (Roughly speaking, the first compares the differential build out of the $(-\circ \delta_j):\Del([\ell-1],[n])\to \Del([\ell],[n])$ with the de Rham differential, the second compares the simplicial maps $(\rho\circ -):\Del([\ell],[n])\to \Del([\ell],[m])$ with the ones in $\BVB(M)$, and the third checks compatibility under smooth maps $g:M'\to M$.)

First, we check that $\BCh(M)_n(\mc E):N(\Z\Delt{n})\to\OdR^\bu(M)\ul$ is a chain map, i.e., it respects the differentials:
\begin{multline}\label{EQU:CHUE-chain-map-proof}
d\Big(\BCh(M)_n(\mc E)(e_{i_0,\dots, i_\ell})\Big)
=
d\bigg( \int_{|e_{i_0,\dots,i_\ell}|}\Chu\Big(\mc E\big|_{M\times |e_{i_0,\dots,i_\ell}|}\Big) \bigg)
\\
=
 \int_{\partial |e_{i_0,\dots,i_\ell}|}\Chu\Big(\mc E\big|_{M\times |e_{i_0,\dots,i_\ell}|}\Big) \Big|_{\partial |e_{i_0,\dots,i_\ell}|}
=
\sum_j (-1)^j \int_{ |e_{i_0,\dots,\widehat{i_j}, \dots,i_\ell}|}\Chu\Big(\mc E\big|_{M\times |e_{i_0,\dots,\widehat{i_j}, \dots,i_\ell}|}\Big) 
\\
=
\sum_j (-1)^j \BCh(M)_n(\mc E)(e_{i_0,\dots, \widehat{i_j}, \dots, i_\ell}))
=
\BCh(M)_n(\mc E)(d(e_{i_0,\dots, i_\ell})),
\end{multline}
where in the second equality we used the integration over the fiber formula $d\big(\int_M \alpha\big) = (-1)^{\dim(M)-1}\int_{\partial M}\alpha|_{\partial M} +(-1)^{\dim(M)}\int_M d\alpha$, and the fact that $\Chu(\tilde{\mc E})$ is $d$-closed for any $\tilde{\mc E}$.

Next, we check that $\BCh(M)$ is a map of simplicial sets. For a morphism $\rho:[n]\to [m]$, inducing the map $\rho_*:N(\Z\Delt{n})\to N(\Z\Delt{m})$, the diagram
\begin{equation}\label{DIAGRAM:CH(M)-is-simplicial-map}
\xymatrix{ \BVB(M)_m \ar^{\BCh(M)_m\hspace{.7in}}[rr] \ar_{\BVB(M)_\rho}[d] && {\Chain^-}(N(\Z\Delt{m}),\OdR^\bu(M)\ul) \ar^{\kappa\mapsto \kappa\circ \rho_*}[d]   \\ 
\BVB(M)_n \ar^{\BCh(M)_n\hspace{.7in}}[rr]&& {\Chain^-}(N(\Z\Delt{n}),\OdR^\bu(M)\ul) }
\end{equation}
commutes, which can be seen as follows. Let $\mc E\in \BVB(M)_m $ and $e_{i_0,\dots, i_\ell}\in N(\Z\Delt{n})$, so that
\begin{multline}\label{EQU:Ch(U)-circ-VB(U)_rho}
(\BCh(M)_n\circ \BVB(M)_\rho(\mc E))(e_{i_0,\dots, i_\ell})
=
\BCh(M)_n((id_M\times \rho)^*\mc E)(e_{i_0,\dots, i_\ell})
\\
=
\int_{|e_{i_0,\dots, i_\ell}|} \Chu(((id_M\times \rho)^*\mc E)|_{M\times |e_{i_0,\dots, i_\ell}|})
=
\int_{|\Delta^\ell|} \Chu((id_M\times (\rho\circ \delta_{j_{n-\ell}}\circ\dots\circ \delta_{j_1}))^*\mc E),
\end{multline}
where in the last equality, we wrote $|e_{i_0,\dots, i_\ell}|=\delta_{j_{n-\ell}}\circ\dots\circ \delta_{j_1}(|\Delta^\ell|)\subseteq |\Delta^n|$ for a partition $\{i_0,\dots, i_\ell\}\sqcup \{j_1,\dots, j_{n-\ell}\}= \{0,\dots, n\}$ of $\{0,\dots, n\}$ with $0\leq i_0< \dots < i_\ell\leq n$ and $0\leq j_1< \dots< j_{n-\ell}\leq n$, as we did on page \pageref{DEF:|e_i0...in|}. Now, the composition $[\ell]\stackrel{\delta_{j_{n-\ell}}\circ\dots\circ \delta_{j_1}}\longrightarrow[n]\stackrel{\rho}\longrightarrow [m]$ can be rewritten in precisely one of the following two ways (via the morphism relations in $\Del$). Either (in case a) $\rho\circ \delta_{j_{n-\ell}}\circ\dots\circ \delta_{j_1}$ is equal to the composition $\varphi\circ \sigma_p:[\ell]\stackrel{\sigma_p} \longrightarrow[\ell-1] \stackrel{\varphi}\longrightarrow [m]$ for some degeneracy $\sigma_p$ and some $\varphi$, or (in case b) $\rho\circ \delta_{j_{n-\ell}}\circ\dots\circ \delta_{j_1}$ is equal to $\delta_{j'_{m-\ell}}\circ\dots\circ \delta_{j'_1}:[\ell]\longrightarrow[m]$ for some $0\leq j'_1< \dots< j'_{m-\ell}\leq m$. We show that in either case diagram \eqref{DIAGRAM:CH(M)-is-simplicial-map} commutes.
\begin{enumerate}
\item[(a)]
In this case, the map on the topological simplicies $|\Delta^\ell|\stackrel{\sigma_p} \longrightarrow|\Delta^{\ell-1}| \stackrel{\varphi}\longrightarrow |\Delta^m|$ factors through the lower dimensional simplex $|\Delta^{\ell-1}|$, so that the integral over $|\Delta^{\ell}|$ on the right hand side of \eqref{EQU:Ch(U)-circ-VB(U)_rho} vanishes. On the other hand, the map $\rho_*:N(\Z\Delt{n})\to N(\Z\Delt{m})$ appearing in \eqref{DIAGRAM:CH(M)-is-simplicial-map}, maps $e_{i_0,\dots, i_\ell}\in N(\Z\Delt{n})$ to a degenerate element, so that $\rho_*(e_{i_0,\dots, i_\ell})=0$ in the normalized cochain complex $N(\Z\Delt{m})$. Thus, \eqref{DIAGRAM:CH(M)-is-simplicial-map} commutes in this case.
\item[(b)]
In this case, if we use $0\leq j'_1< \dots< j'_{m-\ell}\leq m$ to partition $\{0,\dots, m\}=\{i'_0,\dots, i'_\ell\}\sqcup \{j'_1,\dots, j'_{m-\ell}\}$ via the corresponding $0\leq i'_0< \dots < i'_\ell\leq m$, then $\rho_*:N(\Z\Delt{n})\to N(\Z\Delt{m})$ maps $\rho_*:e_{i_0,\dots, i_\ell}\mapsto e_{i'_0,\dots, i'_\ell}$, so that the right hand side of \eqref{EQU:Ch(U)-circ-VB(U)_rho} becomes
\begin{multline*}
\hspace{.5in}\int_{|\Delta^\ell|} \Chu((id_M\times (\delta_{j'_{m-\ell}}\circ\dots\circ \delta_{j'_1}))^*\mc E)
=
\int_{|e_{i'_0,\dots, i'_\ell}|} \Chu(\mc E|_{M\times |e_{i'_0,\dots, i'_\ell}|})
\\
=
\BCh(M)_m(\mc E)(e_{i'_0,\dots, i'_\ell})
=
(\BCh(M)_m(\mc E)\circ \rho_*)(e_{i_0,\dots, i_\ell}).
\end{multline*}
Thus, \eqref{DIAGRAM:CH(M)-is-simplicial-map} commutes in this case as well.
\end{enumerate}

Finally, for a map of smooth manifolds $g:M'\to M$, the diagram
\begin{equation}\label{EQU:ChDel(M)-is-natural-transformation}
\xymatrix{ \BVB(M) \ar^{\BVB(g)}[rr] \ar_{\BCh(M)}[d] && \BVB(M') \ar^{\BCh(M')}[d]   \\ 
\OMdR(M) \ar^{\OMdR(g)}[rr]&& \OMdR(M') }
\end{equation}
commutes, since for a vector bundle $\mc E\in \BVB(M)_n$ over $M\times |\Delta^n|$, and a generator $e_{i_0,\dots, i_\ell}$ of $N(\Z\Delt{n})$, we have
\begin{multline*}
(\OMdR(g)_n\circ \BCh(M)_n(\mc E))(e_{i_0,\dots, i_\ell})
=
g^*\Big(\int_{|e_{i_0,\dots, i_\ell}|}\Chu(\mc E|_{M\times |e_{i_0,\dots, i_\ell}|})\Big)
\\
=
\int_{|e_{i_0,\dots, i_\ell}|}\Chu((g\times id_{|\Delta^n|})^*(\mc E)|_{M'\times |e_{i_0,\dots, i_\ell}|})
=
(\BCh(M')_n\circ \BVB(g)_n (\mc E))(e_{i_0,\dots, i_\ell}).
\end{multline*}
This shows that $\BCh$ is a natural transformation as claimed.
\end{proof}

\section{Chern character via the nerve of the category of vector bundles}\label{SEC:special-affine}

In this section, we define a version of the Chern character that is associated to a sequence of vector bundles and isomorphisms of bundles. This will be done by interpreting such a sequence as a bundle on the manifold crossed with a simplex, so that we may then use the Chern character map from the previous section. 

\begin{definition}\label{DEF:vb-nabla-VB-nabla}
Denote by $\vbn(M)$ the category of complex vector bundles $\mc E=(E,\nabla)$ over a smooth manifold $M$ with connections, whose morphisms $f\in \Mor(\mc E',\mc E)$ consist of bundle isomorphisms $f:E'\to E$ over the identity of $M$, where $f$ has no particular compatibility with the connections. Denote by $\VB(M)=\Nerve(\vbn(M))$ the nerve of this category, i.e., $\VB(M)$ is the simplicial set, whose set of $n$-simplicies $\VB(M)_n$ is given by sequences of bundles with connections $\mc E_j=(E_j,\nabla_j)$ and isomorphisms
\begin{equation}\label{EQU:vecE}
\vec{\mc E} = (\mc E_0\stackrel {f_1} \longrightarrow \mc E_1\stackrel {f_2} \longrightarrow \mc E_2\longrightarrow \dots \longrightarrow \mc E_{n-2}\stackrel {f_{n-1}} \longrightarrow \mc E_{n-1}\stackrel {f_n} \longrightarrow \mc E_n).
\end{equation}

We are also interested in a variation of the above simplicial set, where we forget about the connection. Denote by $\vb(M)$ the category of vector bundles $E$ over $M$ (where no particular connection on $E$ is chosen), and let $\BUN(M)=\Nerve(\vb(M))$ is its nerve.

In particular, we get the simplicial presheaves $\VB$ (respectively $\BUN$) which assigns to each manifold $M$ the nerve of the groupoid of vector bundles with connections (respectively without connections) and isomorphisms between them.
\end{definition}
Note that for each choice of manifold $M$, the map $\vbn(M) \to \vb(M)$ which merely forgets the connection, is a weak equivalence of categories and so induces a weak equivalence of simplicial presheaves as recorded in the following proposition. 
\begin{proposition}\label{REM: VB iso VB nabla}
The natural map of simplicial presheaves, $\VB \to \BUN$, which forgets connection data is a weak equivalence in the global projective model structure on simplicial presheaves.
\end{proposition}

The next definition describes a way of thinking of an $n$-simplex in the nerve $\VB(M)_n$ as a vector bundle with connection over $M\times \Delta^n$, i.e., as an $n$-simplex in $\BVB(M)_n$ from definition \ref{DEF:VB_Delta}, by affinely extending connections on the vertices of $\Delta^n$. Unfortunately, this ``assignment'' from $\VB(M)_n$ to $\BVB(M)_n$ is not a map of simplicial sets (see lemma \ref{LEM:VB-to-BVB}), and we will recover a map of simplicial sets after composing it with the Chern character map for  $\BVB$ (see proposition \ref{PROP:CH:VBN-to-OM}).

\begin{definition}\label{DEF:E-to-EDelta}
To any $\vec{\mc E}= (\mc E_0\stackrel {f_1}\longrightarrow \dots \stackrel {f_n} \longrightarrow \mc E_n)\in \VB(M)_n$ in the nerve where $\mc E_j=(E_j,\nabla_j)$ as in \eqref{EQU:vecE}, we associate a complex vector bundle $\vec{\mc E}_\Delta=(E_\Delta\to M\times \Delta^n,\nabladelta) \in \BVB(M)_n$ as in \eqref{DEF:VB_Delta} as follows. First, to define the vector bundle $E_\Delta\to M\times |\Delta^n|$ we extend one of the vector bundles $E_j\to M$ from $\vec{\mc E}$ by crossing it with $|\Delta^n|$. Explicitly we make the \emph{choice} of extending the zeroth bundle, i.e.,  $E_\Delta:=E_0\times |\Delta^n|\to M\times |\Delta^n|$. To define the connection $\nabladelta$ on $E_\Delta$, first consider for $j=0,\dots, n$ the connections given by transfer of structure from $E_j$ to $E_0$,
\begin{equation}\label{EQU:nabdladeltaj}
\nabladelta_{,j}:=f_1^{-1}\circ \dots \circ f_j^{-1}\circ \nabla_j\circ f_j \circ \dots \circ f_1
\quad\quad\text{on }E_0.
\end{equation}
Then, we define $\nabladelta$ on $E_0\times |\Delta^n|$ as the affine span of these connections over the simplex, i.e.,
\begin{equation}\label{EQU:nabdladelta}
\nabladelta:=d_{|\Delta^n|}+\nabladelta_{,n}+t_1\cdot (\nabladelta_{,0}-\nabladelta_{,1})+t_2\cdot (\nabladelta_{,1}-\nabladelta_{,2})+\dots + t_n\cdot (\nabladelta_{,n-1}-\nabladelta_{,n}).
\end{equation}
Here, $d_{|\Delta^n|}=\sum_i dt_i\frac{\partial}{\partial t_i}$ is the de Rham differential on $|\Delta^n|$, and $d_{|\Delta^n|}$ and $\nabladelta_{,0},\dots, \nabladelta_{,n}$ are all extended to $E_0\times |\Delta^n|\to M\times |\Delta^n|$. Note that the $j$th vertex $|e_j|$ of  $|\Delta^n|$ is precisely at the coordinates $t_1=\dots=t_j=0$ and $t_{j+1}=\dots= t_n=1$ (cf. \eqref{DEF:|e_i0...in|}), for which the restriction of $\nabladelta$ to $E_0\times |e_j|$ becomes $\nabladelta|_{E_0\times |e_j|}=\nabladelta_{,j}$.
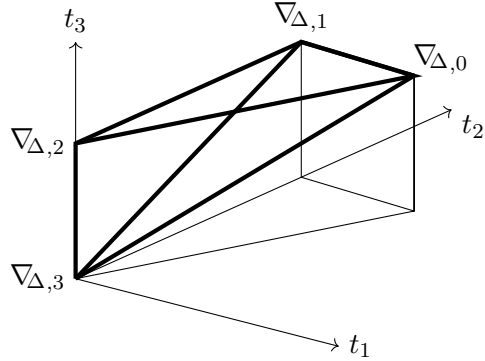
\begin{equation*}
\begin{tikzpicture}[xscale=0.5,yscale=0.45]
\draw [->] (0,0)--(7,-2); \draw [->] (0,0) -- (0,7); \draw [->] (0,0) -- (10,5);
\draw [ultra thick] (0,0)--(0,4)--(6,7)--(9,6);
\draw [ultra thick] (9,6)--(0,0);
\draw [ultra thick] (0,0)--(6,7)--(9,6)--(0,4);
\draw (0,0)--(9,2)--(9,6)--(9,2)--(6,3)--(6,7);
\node [left] at (0,0) {$\nabladelta_{,3}$};
\node [left] at (0,4) {$\nabladelta_{,2}$};
\node [above] at (6,7) {$\nabladelta_{,1}$};
\node [right] at (8.7,6.5) {$\nabladelta_{,0}$};
\node [right] at (7,-2) {$t_1$};
\node [right] at (10,4.5) {$t_2$};
\node [above] at (0,7) {$t_3$};
\end{tikzpicture}
\end{equation*}

\end{definition}

The next lemma shows that the assignment $(\vec{\mc E}\in \VB(M)_\bu)\mapsto (\vec{\mc E}_\Delta \in \BVB(M)_\bu)$ respects the simplicial face and degeneracy maps only up to isomorphism, and so is not an actual map of simplicial sets.

\begin{lemma}\label{LEM:VB-to-BVB}
For $\vec{\mc E}\in \VB(M)_m$ and a morphism $\rho:[n]\to [m]$ of $\Del$, the bundles $(\rho^*(\vec{\mc E}))_\Delta$ and $\rho^*(\vec{\mc E}_\Delta)$ may not be equal. However, there is an isomorphism of vector bundles $F_\rho:(\rho^*(\vec{\mc E}))_\Delta\to \rho^*(\vec{\mc E}_\Delta)$ over $M\times |\Delta^n|$, which commutes with given connections.
\end{lemma}
\begin{proof}
It is enough to show that there are isomorphisms $F_\rho$ for $\rho$ being any degeneracy $\sigma_j:[m+1]\to[m]$ or any face map $\delta_j:[m-1]\to[m]$. Let $\vec {\mc E}=(\mc E_0\stackrel {f_1} \to \dots \stackrel {f_m}\to \mc E_m)\in\VB(M)_m$.

Then $\sigma_j^*(\vec {\mc E})=(\mc E_0\stackrel {f_1} \to \dots\stackrel {f_j}\to\mc E_j \stackrel {id}\to \mc E_j \stackrel {f_{j+1}}\to\dots \stackrel {f_m}\to \mc E_m)$ is given by ``repeating'' the $j$th vector bundle, so that $(\sigma_j^*(\vec {\mc E}))_{\Delta}$ is the bundle on $E_0\times |\Delta^{m+1}|$ with connection given by the affine span of $\nabladelta_{,0},\dots,\nabladelta_{,j},\nabladelta_{,j},\dots, \nabladelta_{,m}$ from \eqref{EQU:nabdladeltaj}. On the other hand, $\vec{\mc E}_\Delta$ is the bundle $E_0\times |\Delta^{m}|$ with connection given by the affine span of $\nabladelta_{,0},\dots,\nabladelta_{,j},\dots, \nabladelta_{,m}$. To obtain $\sigma_j^*(\vec{\mc E}_\Delta)$ we pull back $\vec{\mc E}_{\Delta}$ under the map $id_M\times \sigma_j:M\times|\Delta^{m+1}|\to M\times |\Delta^m|$ where $\sigma_j:|\Delta^{m+1}|\to |\Delta^{m}|$, $\sigma_j(t_1,\dots, t_{m+1})=(t_1,\dots,\widehat{t_{j+1}},\dots t_{m+1})$; cf. \ref{NOTATION:simplicial-sets}. Recall from \eqref{EQU:ith-vertex} that the $i$th vertex $|e_i|\subseteq |\Delta^n|$ is given by the coordinates $t_1=\dots=t_i=0$ and $t_{i+1}=\dots=t_{n}=1$, so that $\sigma_j$ maps $|e_i|$ to $|e_i|$ for $i=0,\dots, j$, and $|e_i|$ to $|e_{i-1}|$ for $i=j+1,\dots, m+1$, which means that under the pullback of $id_M\times \sigma_j$ we get the connection given by the affine span of $\nabladelta_{,0},\dots,\nabladelta_{,j},\nabladelta_{,j},\dots, \nabladelta_{,m}$, just as we did in the case of $(\sigma_j(\vec {\mc E}))_{\Delta}$.

Similarly, for $j=1,\dots, m$, the $j$th face map for is given by ``skipping'' the $j$th vector bundle, $\delta_j^*(\vec {\mc E})=(\mc E_0\stackrel {f_1} \to \dots\stackrel {f_{j-1}}\to\mc E_{j-1} \stackrel {f_{j+1}\circ f_{j}}\to \mc E_{j+1} \stackrel {f_{j+2}}\to\dots \stackrel {f_m}\to \mc E_m)$ so that $(\delta_j^*(\vec {\mc E}))_{\Delta}$ is the bundle on $E_0\times |\Delta^{m-1}|$ with connection given by the affine span of $\nabladelta_{,0},\dots,\nabladelta_{,j-1},\nabladelta_{,j+1},\dots, \nabladelta_{,m}$ from \eqref{EQU:nabdladeltaj}. On the other hand, $\delta_j^*(\vec{\mc E}_\Delta)$ is the pullback of $\vec{\mc E}_{\Delta}$ under the map $id_M\times \delta_j:M\times|\Delta^{m-1}|\to M\times |\Delta^m|$ where $\delta_j:|\Delta^{m-1}|\to |\Delta^{m}|$ given explicitly in \ref{NOTATION:simplicial-sets} maps $|e_i|$ to $|e_i|$ for $i=0,\dots, j-1$ and $|e_i|$ to $|e_{i+1}|$ for $i=j,\dots, m-1$. This means that the pullback of $id_M\times \delta_j$ gives the connection given by the affine span of $\nabladelta_{,0},\dots,\nabladelta_{,j-1},\nabladelta_{,j+1},\dots, \nabladelta_{,m}$, just as in the case of $(\delta_j(\vec {\mc E}))_{\Delta}$.

Finally, the $0$th face is given by $\delta_0^*(\vec {\mc E})=(\mc E_1\stackrel {f_2} \to \dots\to \mc E_{j+1} \stackrel {f_{j+2}}\to\dots \stackrel {f_m}\to \mc E_m)$ so that $(\delta_j^*(\vec {\mc E}))_{\Delta}$ is the bundle on $E_1\times |\Delta^{m-1}|$ with connection given by the affine span of $\nabladelta_{,1}^{[E_1]},\dots,\nabladelta_{,m}^{[E_1]}$, which are only pulled back to $E_1$ and not to $E_0$ (i.e., start with $f_2$ instead of $f_1$ in \eqref{EQU:nabdladeltaj}). On the other hand, $\vec{\mc E}_\Delta$ is the bundle $E_0\times |\Delta^{m}|$ with connection given by the affine span of $\nabladelta_{,0},\dots,\nabladelta_{,m}$, which is pulled back under $id\times\delta_0:M\times |\Delta^{m-1}|\to M\times |\Delta^m|$. Since $\delta_0:|\Delta^{m-1}|\to|\Delta^m|$ maps each $|e_i|$ to $|e_{i+1}|$ for all $i$, we see that $\delta_0^*(\vec{\mc E}_\Delta)$ is the bundle $E_0\times |\Delta^{m-1}|$ with connection given by the affine space of $\nabladelta_{,1},\dots,\nabladelta_{,m}$. The vector bundles and connections of $\delta_0^*(\vec{\mc E}_\Delta)$ and $(\delta_0^*(\vec{\mc E}))_\Delta$ are thus related by pulling back $(\delta_0^*(\vec{\mc E}))_\Delta$ on $E_1\times|\Delta^{m-1}|$ via the map $f_1\times id_{|\Delta^{m-1}|}:\mc E_0\times |\Delta^{m-1}|\to \mc E_1\times|\Delta^{m-1}|$. Since this pullback is an isomorphism of bundles with connection, we obtain the claimed isomorphism.
\end{proof}

Since, by the previous lemma, the assignment  $(\vec{\mc E}\in \VB(M)_\bu)\mapsto (\vec{\mc E}_\Delta \in \BVB(M)_\bu)$ respects the simplicial identities only up to isomorphism, we obtain an actual simplicial set map after composing it with the Chern character.
\begin{definition}\label{DEF:CH-for-VBnabla}
We define a map of simplicial presheaves $\Ch: \VB \to \OMdR$ as follows. Define $\Ch(M)_n:\VB(M)_n\to \OMdR(M)_n$ to be $\Ch(M)_n(\vec{\mc E}):=\BCh(M)_n(\vec{\mc E}_\Delta)$,
\begin{equation}\label{EQU:Ch-for-VB}
\xymatrix{ \VB(M) \ar@{.>}[r]^{\vec{\mc E}\mapsto \vec{\mc E}_\Delta} \ar_{\Ch(M)}@/_1pc/[rr]
& \BVB(M) \ar^{\BCh(M)}[r]  
& \OMdR(M) }
\end{equation}
\end{definition}
The next proposition checks that $\Ch$ is indeed a map of simplicial presheaves.
\begin{proposition}\label{PROP:CH:VBN-to-OM}
The map $\Ch: \VB \to \OMdR$ is a map of simplicial presheaves, i.e., a natural transformation between $\VB$ and $\OMdR$.
\end{proposition}
\begin{proof}
Since, by proposition \ref{PROP:Ch-presheaf-map}, $\BCh(M)_n(\mc E)\in \OMdR(M)_n$ is well defined for any $\mc E$, we only need to check that $\Ch(M)$ is a map of simplicial sets, and  that $\Ch$ is a natural transformation.

For the first, let $\rho:[n]\to [m]$ be a morphism in $\Del$, and $\vec{\mc E}\in \VB(M)_m$, we need to check that $\Ch(M)_n(\rho^*(\vec{\mc E}))= \BCh(M)_n((\rho^*(\vec{\mc E}))_\Delta)$ and  $\rho^*(\Ch(M)_n(\vec{\mc E}))=\rho^*(\BCh(M)_n(\vec{\mc E}_\Delta))=\BCh(M)_n(\rho^*(\vec{\mc E}_\Delta))$ are equal; cf. the corresponding diagram \eqref{DIAGRAM:CH(M)-is-simplicial-map}. By lemma \ref{LEM:VB-to-BVB} the bundles $(\rho^*(\vec{\mc E}))_\Delta$ and $\rho^*(\vec{\mc E}_\Delta)$ are isomorphic as bundles with connections, so that their trace of powers of curvatures are equal, and thus, by \eqref{EQU:CH(M)_n(E)(e)}, applying $\BCh(M)_n$ yields the output in $\OMdR(M)_n$.

To see that $\Ch(M)$ is a natural transformation, note that for a smooth map of manifolds $g:M'\to M$, and $\vec {\mc E}=(\mc E_0\stackrel {f_1} \to \dots \stackrel {f_n}\to \mc E_n)\in \VB(M)_n$, we have $\VB(g)_n(\vec {\mc E})=(\mc E'_0\stackrel {f'_1} \to \dots \stackrel {f'_n}\to \mc E'_n)$ where $\mc E'_j=(E'_j,\nabla'_j)=(g^*(E_j),g^*(\nabla_j))$ and $f'_j=g^*(f_j)$ are the pullbacks under $g$,
\begin{equation*}
\xymatrix{ 
&& E_{j-1} \ar[rr]^{f_j}\ar[rd] && E_j\ar[ld]\\
E'_{j-1} \ar[rr]^{f'_j=g^*(f_j)}\ar[rru]\ar[rd] && E'_j\ar[rru]\ar[ld] &M& \\
& M'\ar[rru]_{g} &&&}
\end{equation*}
so that $(\VB(g)_n(\vec {\mc E}))_\Delta$ is the bundle $g^*(E_0)\times |\Delta^n|\to M'\times |\Delta^n|$ with a connection that is the affine span of the transferred connections ${\nabla'_{\hspace{-1mm}\Delta}}_{,j}:=(f'_1)^{-1}\circ \dots \circ (f'_j)^{-1}\circ \nabla'_j\circ f'_j \circ \dots \circ f'_1=g^*(f_1^{-1}\circ \dots \circ f_j^{-1}\circ \nabla_j\circ f_j \circ \dots \circ f_1)=g^*(\nabladelta_{,j})$ on the $j$th vertex. On the other hand, taking first $\vec{\mc E}_\Delta$ gives the vector bundle $E_0\times |\Delta^n|\to M\times |\Delta^n|$ with connection the affine span of $\nabladelta_{,j}=f_1^{-1}\circ \dots \circ f_j^{-1}\circ \nabla_j\circ f_j \circ \dots \circ f_1$ on the $j$th vertex. Now $g$ acts on this via pullback $g\times id_{|\Delta^n|}:M'\times id_{|\Delta^n|}\to M\times id_{|\Delta^n|}$ to give the bundle $(g\times id_{|\Delta^n|})^*(E_0\times |\Delta^n|)=g^*(E_0)\times |\Delta^n|$ over $M'\times |\Delta^n|$ with the connection $(g\times id_{|\Delta^n|})^*($affine span of $\nabladelta_{,j}$s$)$, which equals the affine span of the connections $g^*(\nabladelta_{,j})$. This shows that the following diagram commutes:
\[
\xymatrix{ \VB(M) \ar^{\VB(g)}[rr] \ar_{\vec{\mc E}\mapsto \vec{\mc E}_\Delta}[d] && \VB(M') \ar^{\vec{\mc E}\mapsto \vec{\mc E}_\Delta}[d]   \\ 
\BVB(M) \ar^{\BVB(g)}[rr]&& \BVB(M') }
\]
Composing this diagram with \eqref{EQU:ChDel(M)-is-natural-transformation} shows that $\Ch=\BCh\circ (-)_\Delta$ is a natural transformation.

This completes the proof of the proposition.
\end{proof}

\begin{lemma}\label{LEM:degrees-of-Chern-from-local}
For a fixed $n$-simplex $\vec{\mc E}= (\mc E_0\stackrel {f_1}\longrightarrow \dots \stackrel {f_n} \longrightarrow \mc E_n)\in \VB(M)_n$, the Chern character maps this to $\BCh(M)_n(\vec{\mc E}_\Delta)(e_{i_0,\dots, i_\ell})$ from \eqref{EQU:CH(M)_n(E)(e)} with de Rham degrees $\geq \ell$:
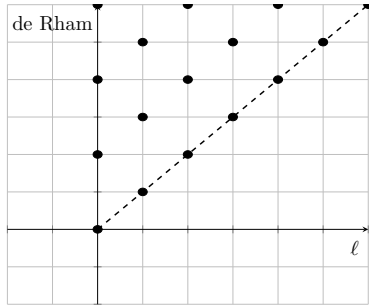
\[
\scalebox{0.7}{
\begin{tikzpicture}
\begin{axis}[axis lines = center, grid=both, xmin=-2, xmax=6, ymin=-2, ymax=6, xtick={-2,...,6},  ytick={-2,...,6}, xticklabels={}, yticklabels={}];
\draw [fill, black] (0,0) circle [radius=.1]; 
\draw [fill, black] (0,2) circle [radius=.1]; 
\draw [fill, black] (0,4) circle [radius=.1]; 
\draw [fill, black] (0,6) circle [radius=.1]; 
\draw [fill, black] (1,1) circle [radius=.1]; 
\draw [fill, black] (1,3) circle [radius=.1]; 
\draw [fill, black] (1,5) circle [radius=.1]; 
\draw [fill, black] (2,2) circle [radius=.1]; 
\draw [fill, black] (2,4) circle [radius=.1]; 
\draw [fill, black] (2,6) circle [radius=.1]; 
\draw [fill, black] (3,3) circle [radius=.1]; 
\draw [fill, black] (3,5) circle [radius=.1]; 
\draw [fill, black] (4,4) circle [radius=.1]; 
\draw [fill, black] (4,6) circle [radius=.1]; 
\draw [fill, black] (5,5) circle [radius=.1]; 
\draw [fill, black] (6,6) circle [radius=.1]; 
\draw [dashed, thick] (0,0)--(6,6);
\node at (5.7,-0.5) {$\ell$};  \node at (-1,5.5) {\textnormal{de Rham}};
\end{axis}
\end{tikzpicture}}
\]
\end{lemma}
\begin{proof}
We will see this in remark \ref{RMK:expand-CH-as-MC-forms} by using local coordinates in the expression \eqref{EQU:CH(M)_n(E)(e)}.
\end{proof}

\section{The Chern character from Chern-Weil theory and from Bott-Tu}\label{SEC:CW-vs-BT}

We now apply the functor from the previous section to open sets coming from a chosen cover $\mc U$ of a manifold $M$. We use the totalization to glue these local pieces together to obtain a Chern character functor that can be applied to vector bundles which are given via a local description. We use this to show how formulas from Chern-Weil theory and local formulas provided by Bott-Tu \cite{BT} for computing the Chern character fit into the above picture. Moreover we examine how these two descriptions can be seen to be equivalent in this formalism.

\subsection{Totalizing the Chern character map}\label{SEC:Tot-of-Chern-CW-BT}

We now recall the \v{C}ech nerve of an open cover $\mc U$ and its totalization, and apply it to the functor $\Ch(.):\VB(.)\to \OMdR(.)$ from \eqref{EQU:Ch-for-VB}. The totalization gives a map of simplicial sets \eqref{EQU:Tot(Ch(NU))-as-sSet}, whose input has $0$-simplicies that are roughly vector bundles given via a local description (see proposition \ref{PROP:Tot(VB)01}), and whose output has $0$-simplicies that land in the \v{C}ech-de Rham complex (see proposition \ref{PROP:Tot(Omega)}).

We start by recalling the relevant definitions and facts, which will  closely follow the descriptions given in \cite{GMTZ1}. While the descrtiption in \cite{GMTZ1} used complex manifolds instead of smooth manifolds, the stated facts here will transfer over to the smooth case, and we refer to \cite{GMTZ1} for full details.

\begin{definition}[{\cite[Def. 3.1, Prop. 3.2]{GMTZ1}}]\label{DEF:Cech-Nerve}
Let $\mc U=\{U_i\}_{i\in I}$ be an open cover of a smooth manifold $M\in \Man$. Then, define the \v{C}ech nerve $\NU:\Del\to \Man^{op}$ for $\mc U$ to be given by $\NU:[k]\mapsto \NU_k:=\coprod\limits_{i_0,\dots,i_k\in I} U_{i_0,\dots,i_k}$, where we use the usual shorthand notation $U_{i_0,\dots,i_k}=U_{i_0}\cap\dots\cap U_{i_k}$. Face maps $d_j:\NU_k\to \NU_{k-1}$ are induced by inclusions of open sets $U_{i_0,\dots,i_k}\stackrel{inc}\hookrightarrow U_{i_0,\dots,\widehat{i_j},\dots,\i_k}$, degeneracies $s_j:\NU_k\to \NU_{k+1}$ are induced identity maps $U_{i_0,\dots,i_k}\stackrel{id}\to U_{i_0,\dots,i_j,i_j,\dots,\i_k}$.

Note, that the compositions $\VB(\NU):\Del\stackrel{\NU}\to\Man^{op}\stackrel{\VB}\to\sSet$ and $\OMdR(\NU):\Del\stackrel{\NU}\to\Man^{op}\stackrel{\OMdR}\to\sSet$ are cosimplicial simplicial sets. The functor $\Ch: \VB \to \OMdR$ induces a map $\Ch(\NU): \VB(\NU) \to \OMdR(\NU)$ between cosimplicial simplicial sets.
\end{definition}

Next, we recall the definition of the totalization.

\begin{definition}[{\cite[Def. D.1]{GMTZ1}}]
For a cosimplicial simplicial set  ${\bf X}:\Del\to \sSet$, we define the totalization $\Tot({\bf X})$ to be the simplicial set which is the equalizer of
\begin{equation}\label{EQU:totalization-equalizer}
\prod_{[\ell]\in \Del} ({\bf X}([\ell]))^{\Delta^\ell}\stackrel[\psi]{\phi}{\rightrightarrows} \prod_{\rho\in\Del([n], [m])} ({\bf X}([m]))^{\Delta^{n}}
\end{equation}
where $\Delta^p$ is the simplicial set whose $q$-simplicies are $(\Delta^p)_q=\Del([q],[p])$, and exponentiation is given by the simplicial set $({\bf X}([q]))^{\Delta^p}=Map({\bf X}([q]),\Delta^p)$ (see \cite[Ex. C.1 (6) and (4)]{GMTZ1}).

Applying the totalization to $\VB(\NU)$, $\OMdR(\NU)$, and the Chern character map $\Ch(\NU)$, induces a map of simplicial sets
\begin{equation}\label{EQU:Tot(Ch(NU))-as-sSet}
\Tot(\Ch(\NU)): \Tot(\VB(\NU)) \to \Tot(\OMdR(\NU))
\end{equation}
\end{definition}

We now describe the totalization of the cosimplicial simplicial sets in definiton \ref{DEF:Cech-Nerve} more explicitly.

\begin{proposition}[{\cite[Prop.s 3.4, 3.15]{GMTZ1}}]\label{PROP:Tot(VB)01}
Let $\mc U=\{U_i\}_{i\in I}$ be an open cover on $M$.
\begin{itemize}
\item
A $0$-simplex in $\Tot(\VB(\NU))$, which we deonte by $\mathfrak E$, consists of vector bundles $E_i\to U_i$ for each $i\in I$ and connections $\nabla_i$ on $E_i$, together with bundle isomorphism $g_{i,j}:E_j|_{U_{i,j}}\to E_i|_{U_{i,j}}$ for each $i,j\in I$ (the $g_{i,j}$ do {\bf not} have to preserve the connections) satisfying $g_{i,j}\circ g_{j,k}=g_{i,k}$ on $U_{i,j,k}$ and $g_{i,i}=id$.
\item
A $1$-simplex in $\Tot(\VB(\NU))$ consists of two $0$-simplicies $\mathfrak E^{(0)}$ and $\mathfrak E^{(1)}$ together with bundle isomorphisms $f_i:E_i^{(0)}\to E_i^{(1)}$ over $U_i$ for each $i\in I$ (the $f_i$ do {\bf not} have to preserve the connections) satisfying $f_i\circ g_{i,j}^{(0)}=g_{i,j}^{(1)}\circ f_j$ on $U_{i,j}$.
\end{itemize}
\end{proposition}
\begin{proof}
The proof is just as in \cite[Prop.s 3.4, 3.15]{GMTZ1}, replacing holomorphic bundle maps and holomorphic connections with complex vector bundle maps and connections.
\end{proof}

\begin{proposition}[{\cite[Prop. 3.13]{GMTZ1}}]\label{PROP:Tot(Omega)}
Let $\mc U=\{U_i\}_{i\in I}$ be an open cover on $M$. Then, there is a map of simplicial sets
\begin{equation}
\Tot(\OMdR(\NU))\to \DKSet(\vC^\bu(\mc U,\OdR^\bu)[u]^{\leq 0})
\end{equation}
where $\DKSet$ is the Dold-Kan functor (cf. definition \ref{DEF:OMdR}) and $\vC^\bu(\mc U,\OdR^\bu)$ is the \v{C}ech-de Rham complex. In particular, we have the following:
\begin{itemize}
\item The $0$-simplicies $\DKSet(\vC^\bu(\mc U,\OdR^\bu)[u]^{\leq 0})_0\cong \vC^\bu(\mc U,\OdR^\bu)^{even}$ are given by elements of the \v{C}ech-de Rham complex of even degree.
\item A $1$-simplex in $\DKSet(\vC^\bu(\mc U,\OdR^\bu)[u]^{\leq 0})_1$ is given by two $0$-simplicies as in the previous point, $\alpha^{(0)}, \alpha^{(1)}\in \vC^\bu(\mc U,\OdR^\bu)^{even}$, together with an element $\beta\in \vC^\bu(\mc U,\OdR^\bu)^{odd}$ such that its \v{C}ech-de Rham differential of $\beta$ is the difference $\alpha^{(1)}-\alpha^{(0)}$.
\end{itemize}
\end{proposition}
\begin{proof}
Just as in \cite[Prop. 3.13]{GMTZ1}, $\Tot(\OMdR(\CN \mc U))=\Tot(\DKSet (\OdR^\bu(\CN \mc U)\ul))$, and we have maps of simplicial sets 
\begin{multline*}
\Tot(\DKSet( \OdR^\bu(\CN \mc U)\ul))
\cong \Tot(\mc F(\DK( \quot(\OdR^{\bu,\bu}))))
\cong \mc F(\Tot(\DK( \quot(\OdR^{\bu,\bu}))))
\\
\to \mc F(\DK(\Tot (\quot(\OdR^{\bu,\bu}))))
\cong \DKSet(\quot(\tot (\OdR^{\bu,\bu})))
\cong\DKSet(\vC^\bu(\mc U, \OdR^\bu)\ul)
\end{multline*}
Here, it is worth noting that the lemmas 3.10, 3.11, and 3.12 in \cite{GMTZ1} all allow for the case where $\OdR^\bu$ has a non-vanishing differential, whereas the differential in $\Om^\bu_{hol}$ in \cite{GMTZ1} was taken to be zero. Note that, by \cite[Def. 3.8]{GMTZ1}, $\tot (\OdR^{\bu,\bu})$ in total degree $p$ is given by $\tot (\OdR^{\bu,\bu})^p \cong\bigoplus\limits_{n}\prod\limits_{i_0,\dots,i_n}\OdR^{p-n}(U_{i_0,\dots,i_n})\cong \vC^\bu(\mc U, \OdR^\bu)^p$.

For the last statements, note that the $0$-simplicies of $\DKSet(\vC^\bu(\mc U,\OdR^\bu[u])^{\leq 0})$ are given by $(\vC^\bu(\mc U,\OdR^\bu)[u])^0\stackrel \cong\to  C^\bu(\mc U,\OdR^\bu)^{even},\sum_{p\geq 0}\omega^{2p}\cdot u^p \mapsto \sum_{p\geq 0}\omega^{2p}$, where $\omega^q$ means an element of degree $q$ in $\vC^\bu(\mc U,\OdR^\bu)^q$. The statement for the $1$-simplicies comes from the description of the Dold-Kan functor in definition \ref{DEF:OMdR}. Note that we also have the identification $(\vC^\bu(\mc U,\OdR^\bu)[u])^{-1}\stackrel \cong\to  C^\bu(\mc U,\OdR^\bu)^{odd},\sum_{p\geq 1}\omega^{2p-1}\cdot u^p \mapsto \sum_{p\geq 1}\omega^{2p-1}$.
\end{proof}

\begin{proposition}\label{PROP:Chern-for-choice-of-cover}
Let $\mc U=\{U_i\}_{i\in I}$ be an open cover of a manifold $M$. For $i\in I$, let $E_i\to U_i$ be vector bundles, and for $i,j\in I$ let $g_{i,j}:E_j|_{U_{i,j}}\to E_i|_{U_{i,j}}$ be bundle isomorphisms satisfying $g_{i,j}\circ g_{j,k}=g_{i,k}$ on $U_{i,j,k}$ and $g_{i,i}=id$.
\begin{itemize}
\item
A choice of connections $\nabla_i$ on $E_i$, $\forall i\in I$, gives a $0$-simplex in $\Tot(\VB(\NU))$. Using the totalization of the Chern character map $\Tot(\Ch(\NU)): \Tot(\VB(\NU)) \to \Tot(\OMdR(\NU))$, is mapped to the (even) \v{C}ech-de Rham forms $\{\alpha_{i_0,\dots,i_n}\in \OdR^\bu(U_{i_0,\dots, i_n})\}$ given by
\begin{equation}\label{EQU:Tot-0-Cech-deRham}
\alpha_{i_0,\dots,i_n}=\sum_{k\geq 0} \frac{1}{k!\cdot (2\pi i)^k}\int_{|\Delta^n|}\tr\big(R^k\big)
\end{equation}
Here, for a fixed choice of indicies $i_0,\dots, i_n$, $R$ denotes the curvature of the connection \eqref{EQU:nabdladelta} on $U_{i_0,\dots, i_n}\times \Delta^n$ given by the sequence of vector bundles $\vec{\mc E}= (\mc E_{i_0}\stackrel {g_{i_1,i_0}}\longrightarrow \dots \stackrel {g_{i_n,i_{n-1}}} \longrightarrow \mc E_{i_n})$ where $\mc E_{p}=(E_{p|_{U_{i_0,\dots,i_n}}}\to U_{i_0,\dots,i_n},\nabla_p|_{U_{i_0,\dots,i_n}})$.
\item
For two choices of connections $\nabla_i^{(0)}$ and $\nabla_i^{(1)}$ on $E_i$, $\forall i\in I$, the identity map $f_i:E_i\to E_i$ gives a $1$-simplex in $\Tot(\VB(\NU))$. Therefore, the difference of the two \v{C}ech-de Rham forms $\{\alpha^{(0)}_{i_0,\dots,i_n}\}$ and $\{\alpha^{(1)}_{i_0,\dots,i_n}\}$ from the previous point are exact.
\end{itemize}
\end{proposition}
\begin{proof}
The totalization in terms of the left hand-side of \eqref{EQU:totalization-equalizer} is given by sequences of bundles (cf. \cite[Prop. 3.4]{GMTZ1}):
\[
\xymatrix{
{ \begin{matrix} \text{on }U_i: \\ E_i \end{matrix}} & 
{ \begin{matrix} \text{on }U_{i,j}: \\ E_i\to E_j \end{matrix}} & 
{ \begin{matrix} \text{on }U_{i,j,k}: \\ E_i\to E_j\to E_k \end{matrix}} &
{ \begin{matrix} \dots \\ \dots\end{matrix}} 
}
\]
Writing these sequences of bundles on $U_{i_0,\dots, i_n}$ as a bundle on $U_{i_0,\dots, i_n}\times \Delta^n$ as in definition \ref{DEF:E-to-EDelta}, the Chern character assigns the forms from \eqref{EQU:CH(M)_n(E)(e)} to this bundle, which is precisely \eqref{EQU:Tot-0-Cech-deRham}. Here, we removed the powers of $u$ as we did in the identification of $0$-simplicies with even \v{C}ech-de Rham forms in proposition \ref{PROP:Tot(Omega)}. (Note that $\int_{\Delta^n}\tr(R^k)$ has a total degree of (\v{C}ech degree)+(de Rham degree) $=n+(2k-n)=2k$.)

The second bullet point follows from the second bullet points in propositions \ref{PROP:Tot(VB)01} and \ref{PROP:Tot(Omega)}.
\end{proof}

The following two examples apply the above proposition to two particular choices of connections on bundles $E_i\to U_i$.

\begin{example}[Chern-Weil formula]\label{EXA:Chern-Weil-example}
Let $E\to M$ be a vector bundle of rank $d$ over $M$. Let $\mc U=\{U_i\}_{i\in I}$ be a good open cover of $M$ so that we can trivialize $E$ over $\mc U$, i.e., we have trivializations $s_i:E_i:=\C^d\times U_i \stackrel\cong\to E|_{U_i}$. Denote by $g_{i,j}:=(s_i^{-1}|_{U_{i,j}})\circ (s_j|_{U_{i,j}}):E_j|_{U_{i,j}}\to E|_{U_{i,j}}\to E_i|_{U_{i,j}}$.

Let $\nabla$ be a (global) connection on $E$ (which always exists using a partition of unity construction), and denote by $\nabla_i=s_i^{-1} \circ \nabla|_{U_i}\circ s_i$ the connection transferred to $E_i$, so that $\mc E_i=(E_i,\nabla_i)$. Now, fix indices $i_0,\dots, i_n$. We get a sequence of vector bundles $\vec{\mc E}= (\mc E_{i_0}\stackrel {g_{i_1,i_0}}\longrightarrow \dots \stackrel {g_{i_n,i_{n-1}}} \longrightarrow \mc E_{i_n})$ on $U_{i_0,\dots, i_n}$ as in definition \ref{DEF:E-to-EDelta}. Since all connections come from the global connection $\nabla$, trasnfering the connection as in \eqref{EQU:nabdladelta} gives $\nabla_{\Delta,j}-\nabla_{\Delta,j-1}=0$, and so $\nabla_\Delta=d_{|\Delta^n|}+\nabla_{\Delta,n}=d_{|\Delta^n|}+\nabla|_{U_{i_0,\dots,i_n}}$ is the connection on $E_\Delta=E_0\times |\Delta^n|$ in \eqref{EQU:nabdladelta}. In particular, we get that the curvature is $R=(\nabla|_{U_{i_0,\dots,i_n}})^2$ is the curvature coming from $\nabla$, which is independent of the parameters $t_i$ in the simplex $|\Delta^n|$. It follows that the integral in \eqref{EQU:Tot-0-Cech-deRham} vanishes for $n>0$, and our Chern character forms are $\alpha_{i_0,\dots,i_n}=0$ for $n>0$, while $\alpha_{i}=\sum_{k\geq 0}\frac{1}{k!(2\pi i)^k} \tr(R^k)$ is the usual Chern character from Chern-Weil theory. Thus, the output \eqref{EQU:Tot-0-Cech-deRham} is concentrated in \v{C}ech degree $0$.
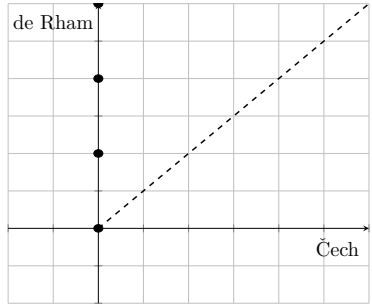
\[
\scalebox{0.7}{
\begin{tikzpicture}
\begin{axis}[axis lines = center, grid=both, xmin=-2, xmax=6, ymin=-2, ymax=6, xtick={-2,...,6},  ytick={-2,...,6}, xticklabels={}, yticklabels={}];
\draw [fill, black] (0,0) circle [radius=.1]; 
\draw [fill, black] (0,2) circle [radius=.1]; 
\draw [fill, black] (0,4) circle [radius=.1]; 
\draw [fill, black] (0,6) circle [radius=.1]; 
\draw [dashed, thick] (0,0)--(6,6);
\node at (5.3,-0.5) {\v{C}ech};  \node at (-1,5.5) {de Rham};
\end{axis}
\end{tikzpicture}}
\]
Note, that the forms $\al_i$ are \v{C}ech-closed and therefore glue to global forms on $M$, which is given by the global Chern character forms for $(E,\nabla)$.
\end{example}

\begin{example}[Bott-Tu formula]\label{EXA:Bott-Tu-example}
Let $E\to M$ be a vector bundle of rank $d$ over $M$. Let $\mc U=\{U_i\}_{i\in I}$ be a good open cover of $M$ so that we can trivialize $E$ over $\mc U$, i.e., we have trivializations $s_i:E_i:=\C^d\times U_i \stackrel\cong\to E|_{U_i}$. Denote by $g_{i,j}:=(s_i^{-1}|_{U_{i,j}})\circ (s_j|_{U_{i,j}}):E_j|_{U_{i,j}}\to E|_{U_{i,j}}\to E_i|_{U_{i,j}}$.

On each $E_i=\C^d\times U_i $, we choose the flat connection $\nabla_i:=d$ given by the de Rham differential. We claim that in this case the connection is the one given by Bott and Tu \cite[p. 305]{BT}. In more detail,  in contrast to \eqref{EQU:standard-n-simplex} we parametrize $|\Delta^n| $ as $|\Delta^n| = \{(\widetilde{t_0},\dots,\widetilde{t_n})\in \R^n\,|\, \widetilde{t_j}\geq 0, \sum \widetilde{t_j}=1\}$. These two representations can be related via $\widetilde{t_j}=t_{j+1}-t_{j}$ (where we set $t_0=0$ and $t_{n+1}=1$). Now for a fixed sequence of vector bundles $\vec{\mc E}= (\mc E_{i_0}\stackrel {g_{i_1,i_0}}\longrightarrow \dots \stackrel {g_{i_n,i_{n-1}}} \longrightarrow \mc E_{i_n})$ on $U_{i_0,\dots, i_n}$, we evaluate the transferred connections \eqref{EQU:nabdladeltaj} to be $\nabla_{\Delta,j}=g_{i_j,i_0}^{-1}\circ d_U\circ g_{i_j,i_0}=g_{i_j,i_0}^{-1} \circ d_U(g_{i_j,i_0})+d_U$, where we wrote $d_U$ for the de Rham differential on $U_{i_0,\dots,i_n}$.
With this, we compute the connection \eqref{EQU:nabdladelta} on $U_{i_0,\dots, i_n}\times |\Delta^n|$ to be
\begin{multline}
\nabla_\Delta
=d_{|\Delta^n|}+\nabla_{\Delta,n}+\sum_{j=1}^n t_j\cdot (\nabla_{\Delta,j-1}-\nabla_{\Delta,j})
\\
=d_{|\Delta^n|}+g_{i_n,i_0}^{-1} \circ d_U(g_{i_n,i_0})+d_U+\sum_{j=1}^n t_j\cdot \left(g_{i_{j-1},i_0}^{-1} \circ d_U(g_{i_{j-1},i_0})-g_{i_j,i_0}^{-1} \circ d_U(g_{i_j,i_0})\right)
\\
=d+\sum_{j=0}^n (t_{j+1}-t_j) \cdot g_{i_0,i_j}^{-1} d(g_{i_0,i_j})
=d+\sum_{j=0}^n \widetilde{t_j} \cdot g_{i_0,i_j}^{-1} d(g_{i_0,i_j})=d+\theta_I,
\end{multline}
where we have written $d$ for the de Rham differential on $U_{i_0,\dots,i_n}\times |\Delta^n|$ and (following Bott and Tu \cite[p. 305]{BT}) $\theta_I:=\sum_{j=0}^n \widetilde{t_j} \cdot g_{i_0,i_j}^{-1} d(g_{i_0,i_j})$. Note that this has curvature $R=(d+\theta_I)^2=d(\theta_I)+\frac 1 2 [\theta_I,\theta_I]$, and so the Chern character associates \eqref{EQU:Tot-0-Cech-deRham} to this, i.e., $\alpha_{i_0,\dots,i_n}=\sum_{k\geq 0} \frac{1}{k!\cdot (2\pi i)^k}\int_{|\Delta^n|}\tr\left(R^k\right)$. In particular, in \v{C}ech degree $0$, $\alpha_{i_0}=\tr(Id)=\text{rank}(E)$ vanishes in positive de Rham degrees, since $\nabla_\Delta=d$ is flat, whereas for $n\geq 1$, $\alpha_{i_0,\dots,i_n}$ is of de Rham degree $\geq n$, since $d(\theta_I)$ is a sum of terms with at most one $d\widetilde{t_j}$ factor, and so the integral over $|\Delta^n|$ is non-vanishing only when $d(\theta_I)$ is multiplied exactly $n$ times (but possibly multiplied by further terms of $\frac 1 2 [\theta_I,\theta_I]$).
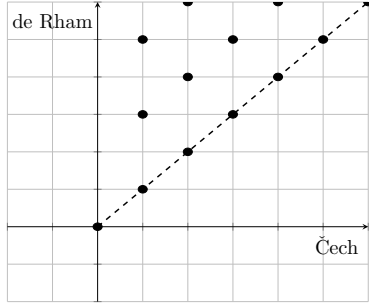
\[
\scalebox{0.7}{
\begin{tikzpicture}
\begin{axis}[axis lines = center, grid=both, xmin=-2, xmax=6, ymin=-2, ymax=6, xtick={-2,...,6},  ytick={-2,...,6}, xticklabels={}, yticklabels={}];
\draw [fill, black] (0,0) circle [radius=.1]; 
\draw [fill, black] (1,1) circle [radius=.1]; 
\draw [fill, black] (1,3) circle [radius=.1]; 
\draw [fill, black] (1,5) circle [radius=.1]; 
\draw [fill, black] (2,2) circle [radius=.1]; 
\draw [fill, black] (2,4) circle [radius=.1]; 
\draw [fill, black] (2,6) circle [radius=.1]; 
\draw [fill, black] (3,3) circle [radius=.1]; 
\draw [fill, black] (3,5) circle [radius=.1]; 
\draw [fill, black] (4,4) circle [radius=.1]; 
\draw [fill, black] (4,6) circle [radius=.1]; 
\draw [fill, black] (5,5) circle [radius=.1]; 
\draw [fill, black] (6,6) circle [radius=.1]; 
\draw [dashed, thick] (0,0)--(6,6);
\node at (5.3,-0.5) {\v{C}ech};  \node at (-1,5.5) {de Rham};
\end{axis}
\end{tikzpicture}}
\]
\end{example}

Using the second bullet point from proposition \ref{PROP:Chern-for-choice-of-cover}, we have the following corollary.
\begin{corollary}\label{COR:CW-id-BT}
The Chern character forms obtained in examples \ref{EXA:Chern-Weil-example} (coming from Chern-Weil theory) and \ref{EXA:Bott-Tu-example} (as found in Bott and Tu \cite[p. 305]{BT}) differ by an exact \v{C}ech-de Rham element.
\end{corollary}

\begin{remark}\label{REM:smooth-d-or-no-d}
From propositions \ref{PROP:Tot(Omega)} we know that the output of $0$-simplicies of $\Tot(\Ch(\NU)): \Tot(\VB(\NU)) \to \Tot(\OMdR(\NU))$ can be taken to land in (even) \v{C}ech-de Rham forms $\DKSet(\vC^\bu(\mc U,\OdR^\bu)[u]^{\leq 0})_0\cong \vC^\bu(\mc U,\OdR^\bu)^{even}$, while these forms for two $0$-simplexes connected by a $1$-simplex are \v{C}ech-de Rham exact. Using the \v{C}ech-de Rham isomorphism $H^\bu(\vC^\bu(\mc U,\OdR^\bu),\delta+d)\cong H^\bu_{\text{dR}}(M,\C)$, isomorphism classes of complex vector bundles thus have an output in de Rham cohomology.

We note that  in definition \ref{DEF:OMdR}, we could also have taken the differential in $\OdR^\bu(-)$ to be $0$ instead of the de Rham differential $d$, which is similar to what we did in a holomorphic setting in our previous paper \cite{GMTZ1} where we used a zero differential on holomorphic forms and restricted ourselves to ``diagonal terms'', i.e. to \v{C}ech forms with equal \v{C}ech and form degree. However, if we take de Rham forms with zero differential, the output from the corresponding Chern character map would be in $ \vC^\bu(\mc U,\OdR^\bu)$ with only the \v{C}ech differential $\delta+0=\delta$. In this case, our map would assign to a complex vector bundles an element in this cohomology which is $H^p(\vC^\bu(\mc U,\OdR^\bu),\delta)= \OdR^p(M)$, since each $\OdR^p$ is an acyclic sheaf.

A similar analysis in the holomorphic setting will be considered in section \ref{SEC:holomorphic-Chern-extended}. In the holomorphic case, having no differential on holomorphic forms gives an interesting cohomology (Hodge cohomology); indeed this was the setup in \cite{GMTZ1}. However, ``turning on the differential'' will also be of interest; see remark \ref{REM:smooth-del-or-no-del} below.
\end{remark}

\subsection{Representations of the Chern Character Map}

The weak equivalence $\VB \xrightarrow{\sim}  \BUN$ of proposition \ref{REM: VB iso VB nabla} means that the Chern character map $\Ch: \VB \to \OMdR$ from definition \ref{DEF:CH-for-VBnabla} induces a Chern character map in the homotopy category of simplicial presheaves from $\BUN$ to $\OMdR$. The examples from the previous section recall the independence of the Chern character of certain choices in representing the map. We now give an interpretation for these examples by providing various representations of the Chern character in terms of sheafified versions of a variety of categories.

\begin{definition}
Denote by $\fvbn(M)$ the category of complex vector bundles $\mc E=(E,\nabla)$ over $M$ with flat connections, whose morphisms $f\in \Mor(\mc E',\mc E)$ consist of bundle isomorphisms $f:E'\to E$ over the identity of $M$, where $f$ has no particular compatibility with the connections. Denote by $\FVB(M)=\Nerve(\fvbn(M))$ the nerve of this category. Elements of $\FVB(M)_n$ are just as elements \eqref{EQU:vecE} in $\VB(M)_n$ with the extra condition that the connections are all flat.
\end{definition}

\begin{definition}
Denote by $\pbn(M) \hookrightarrow \fvbn(M)$ the full subcategory of complex product vector bundles $\mc E=(M \times \mathbb{C}^m,d_{DR})$ over $M$ with the de Rham differential acting as a flat connection, whose morphisms $f\in \Mor(M \times \mathbb{C}^m,M \times \mathbb{C}^m)$ consist of bundle automorphisms given by $\tilde{f}\in \Gamma(M, GL_m(\mathbb{C}))$. Denote by $\PVB(M)=\Nerve(\pbn(M)) \hookrightarrow \FVB(M)$ the nerve of this category.
\end{definition}

Note in particular that each connected component of this subcategory $\pbn(M) \hookrightarrow \fvbn(M)$ has one object and so simplices in the nerve $\PVB(M)$ consist of concatenating automorphisms. 

\begin{definition}\label{COR flat chern}
Since flat vector bundles over $M$ further include in all vector bundles over $M$, $\fvbn(M)\hookrightarrow \vbn(M)$, we obtain induced Chern character maps, 
\begin{equation}\label{EQU:Ch-for-FVB-PrVB}
\begin{tikzcd}
\PVB(M) \arrow[r, hook] \arrow[rrrr, bend left=-25, "\Ch_{pr}(M)"]  & \FVB(M)  \arrow[r, hook] \arrow[rrr, bend left=-25, "\Ch_{fl}(M)"] & \VB(M) \arrow[r, dotted, "\vec{\mc E}\mapsto \vec{\mc E}_\Delta"] \arrow[rr, bend left=-25, "\Ch(M)"] & \BVB(M) \arrow[r, "\BCh(M)"]   & \OMdR(M)
\end{tikzcd}
\end{equation}
which we denote by $\Ch_{pr}$ and $\Ch_{fl}$, respectively, and which are maps of simplicial presheaves.
\end{definition}

Neither $\FVB\hookrightarrow \VB$ nor $\PVB\hookrightarrow \FVB$ are essentially surjective maps, yet we will show that these maps induce equivalences after sheafification. What is meant by \emph{sheafification} is described in \cite[Section 5]{GMTZ2}, but the idea is briefly reviewed below. 

Given a projectively fibrant simplicial presheaf ${\bf F} \in sPre(\Man)$ consider its fibrant replacement in the local projective model structure ${\bf F} \simeq {\bf F}' \in sPre(\Man)_{loc}$, which takes the projective model structure on simplicial presheaves and inverts all maps induced by hypercoverings. By \cite[Proposition 5.2]{GMTZ2}, when ${\bf F}$ is an $n$-type, the fibrant replacement can be computed on a test object $X \in \Man$ by taking the colimit of the simplicial mapping space $\underline{sPre}(\mathcal{U}, {\bf F})$ over all \v{C}ech covers $(\mathcal{U} \to rX) \in \check{S}$, 
\begin{equation}\label{EQ: sheafification def}
 \CechSh{\bf F}(X) := \colim_{(U_{\bullet} \to X) \in \check{S}} \Tot\left( {{\bf F}}(\CechNerve \mathcal{U}  ) \right),
 \end{equation}
where the totalization, $\Tot$, is described in \cite[Appendix D.1 and Proposition 3.15]{GMTZ1}
Furthermore this construction can be applied to a map $\alpha: {\bf F} \to {\bf G}$ between projectively fibrant $n$-types to provide a fibrant approximation, $\CechSh{\alpha}: \CechSh{\bf F} \to \CechSh{\bf F}$. Since this paper is concerned with special simplicial presheaves for the domain of the Chern maps, namely the nerve of some groupoid, these $1$-types satisfy the conditions of the above fibrant approximation construction.

\begin{proposition}
The simplicial presheaves, $\BUN$, $\VB$, and $\OMdR$ are all simplicial sheaves. 
\begin{proof}
First note that all three simplicial presheaves are in fact fibrant in the global projective model structure since on each object they take values in Kan complexes. In the case of $\BUN$ and $\VB$, this is due to taking the nerve of a groupoid whereas in $\OMdR$ it is due to the Dold-Kan map taking values in simplicial abelian groups. 

Next, the simplicial presheaves need to satisfy descent \cite[Definition 4.3]{DHI} with respect to any cover. Since $\VB$ and $\BUN$ are equivalent (proposition \ref{REM: VB iso VB nabla}) it suffices to show that $\BUN$ and $\OMdR$ satisfy descent. Let $M$ be a fixed (paracompact) manifold of dimension $n$. Since $\BUN(M)$ is a 1-type by virtue of it being the nerve of a groupoid, and $\OMdR$ is an $n$-type by virtue of the de Rham complex being trivial after degree $n$, these simplicial presheaves satisfy descent with respect to all hypercovers of $M$ precisely when they satisfy descent with respect to all \v{C}ech cover \cite[Corollary A.8]{DHI}. Thus it suffices to check that ${\bf F}=  \BUN, \OMdR$ satisfies descent for a given \v{C}ech covers, $\{U_i \to M\}$, by showing that the map of simplicial sets
\[ {\bf F}(M) = \underline{sPre}(rM, {\bf F}) \to \underline{sPre}( \vC (\mathcal{U}) , {\bf F}) \simeq \Tot\left( {\bf F}(U_i) \right) \]
is a weak equivalence, where the left hand equality is given by the simplicial Yoneda lemma and the right sided weak equivalence is given by \cite[Theorem 18.7.4]{H}. Proving the weak equivalence for ${\bf F}= \BUN$ is done in three steps. First, the totalization of $\BUN$, which is 2-coskeletal, applied to a cover will again be 2-coskeletal, and thus a 1-type, so only the equivalences for $\pi_0$ and $\pi_1$ are needed. Proving the equivalence for $\pi_0$ boils down to recalling that globally defined bundles are (up to isomorphism) precisely locally defined bundles glued together with isomorphisms satisfying the cocycle condition. Proving the equivalence for $\pi_1$ essentially amounts to recalling that isomorphisms of bundles, $f: E \to E'$,  are equivalent to locally defined isomorphisms $f_i: E_i \to E_i'$ which satisfy strictly commutative squares. For the proof of the weak equivalence when ${\bf F} = \OMdR$, apply the the chain homotopy equivalence between the de Rham complex and the \v{C}ech-de Rham complex (see \cite{BT} for example). 
\end{proof}
\end{proposition}

Next, the maps $\PVB \to \FVB \to \VB \to \BUN$ from \eqref{EQU:Ch-for-FVB-PrVB} are seen to be equivalences in the local projective model structure on simpicial presheaves, despite the two left-most maps coming from ``inclusions'' in the global projective model structure.

\begin{proposition}
Applying the above sheafification functor to the diagram
\begin{equation}\label{Diagram: four bundle presheaves}
\begin{tikzcd}
\PVB \arrow[r, "i_a", hook] \arrow[dr]   & \FVB \arrow[r, "i_b", hook] \arrow[d] & \VB \arrow[dl]\\
 & \BUN  & 
\end{tikzcd}
\end{equation}
where the vertical maps simply forget the connection data, induces weak equivalences between simplicial sheaves
\[\CechSh{\PVB} \simeq \CechSh{\FVB} \simeq \CechSh{\VB} \simeq  \CechSh{\BUN} \simeq \VB \simeq \BUN \]  in the projective model structure, and so all of the simplicial presheaves in \eqref{Diagram: four bundle presheaves} are weakly equivalent in the local projective model strucutre.
\begin{proof}
Each of the simplicial presheaves in \eqref{Diagram: four bundle presheaves} are 1-types (as they are the nerve of a groupoid) and thus their fibrant replacement in the local projective model structure can be computed by \eqref{EQ: sheafification def} as described in the beginning of this section. Next, again since these simplicial presheaves are 2-coskeletal, their sheafification is again 2-coskeletal and thus a 1-type. As such, it suffices to establish the isomorphisms on $\pi_0$ and $\pi_1$. The proof for the horizontal arrows in \eqref{Diagram: four bundle presheaves} inducing an isomorphism on $\pi_0$ essentially follows from the fact that for each vertex in $\CechSh{\VB}(M)$, i.e. locally defined bundles glued by cocycle data, there exists a refinement of the open cover so that the restriction of the locally defined bundles to the finer open sets are now flat bundles. Up to isomorphism on each open set in $\FVB$, a flat connection can then be chosen. Similarly, by refining the open sets further, the bundle will be isomorphic to a trivial product bundle with zero connection. The proof for $\pi_1$ is similar.
\end{proof}
\end{proposition}

\begin{corollary}
Any paths through the following diagram beginning at $\BUN$ and ending at $\OM$, allowing for zig-zags when weak equivalences are present, 
\begin{equation}\label{EQ representations of CH}
\begin{tikzcd}
& \PVB \arrow[ddl, "\check{\dagger}"'] \arrow[r, "i_a", hook]  & \FVB   \arrow[ddl, "\check{\dagger}"'] \arrow[r, "i_b", hook] & \VB  \arrow[ddl, "\check{\dagger}"', "\sim"' {anchor=north, rotate=60}]  \arrow[d, "\sim" {anchor=south, rotate=90}] \arrow[rr, "\Ch"] & & \OM  \arrow[ddl, "\check{\dagger}"' , "\sim"' {anchor=north, rotate=60}] \\
& &  & \BUN  \arrow[ddl, "\check{\dagger}"', "\sim" {anchor=north, rotate=60}, pos=0.2]  \\
\CechSh{\PVB} \arrow[r, "\CechSh{i_a}", "\sim"', hook]  & \CechSh{\FVB} \arrow[r, "\CechSh{i_b}", "\sim"', hook] & \CechSh{\VB} \arrow[d, "\sim" {anchor=south, rotate=90}] \arrow[rr, "\CechSh{\Ch}", crossing over] & & \CechSh{\OM}\\
&  & \CechSh{\BUN}  
\end{tikzcd}
\end{equation}
is a representation of the Chern character map in the homotopy category of simplicial presheaves.
\end{corollary}

The above diagram and corollary of course could be expanded to include more paths (i.e. factoring through sheafification of maps forgetting connections $\CechSh{\PVB} \to \CechSh{{\bf{PVB}}}$ or trivial vector bundles with varying connections, etc).
\begin{example}
The Chern character given by Chern-Weil (example \ref{EXA:Chern-Weil-example}) is represented by the path through the diagram in \eqref{EQ representations of CH}, 
\[\begin{tikzcd}
\BUN \arrow[r, "\sim"', dotted] & \VB  \arrow[r, "\check{\dagger}", "\sim"'] & \CechSh{\VB}  \arrow[r, "\CechSh{\Ch}" ] & \CechSh{\OM} \arrow[r, dotted, "\sim"'] & \OM.
\end{tikzcd}\]
The left-most dotted arrow represents choosing a global connection, the $\check{\dagger}$ arrow is choosing a cover and using the cocycle data, and the last dotted arrow amounts to the \v{C}ech-de Rham isomorphism. On the other hand, the Bott-Tu Chern character (example \ref{EXA:Bott-Tu-example}) is represented by the path,
\[\begin{tikzcd}
\BUN \arrow[r, dotted, "\sim"'] & \CechSh{\PVB}  \arrow[r,  hook, "\CechSh{i_b} \circ \CechSh{i_a}", "\sim"'] & \CechSh{\VB}  \arrow[r, "\CechSh{\Ch}" ] & \CechSh{\OM} \arrow[r, dotted,  "\sim"'] & \OM
\end{tikzcd}\]
The left-most dotted arrow represents choosing a cover which fully trivializes a bundle and using the exterior differential connection, the arrow labeled $\CechSh{i_b} \circ \CechSh{i_a}$ includes this cocycle data into the larger sheaf, and then proceed the same as above. The fact that these two maps are equivalent in the homotopy category implies in particular that they provide equivalent maps $\pi_0( \BUN) \xrightarrow{\pi_0 (\CechSh{\Ch})} \pi_0(\OM)= H^{\bullet}(\OdR^\bu)$ which requires that for a given vector bundle the two classes constructed are cohomologous.
\end{example}

\section{Extending the holomorphic Chern character}\label{SEC:holomorphic-Chern-extended}

In this section, we show how the above structures can be applied in the holomorphic setting that we have explored in \cite{GMTZ1}. After applying the totalization, this yields an extended class that is not closed under the \v{C}ech differential $\delta$ as in \cite{GMTZ1}, but in fact closed under $\delta+\partial$; cf. remarks \ref{REM:smooth-del-or-no-del} and \ref{REM:lowest-terms-extended}. The formulae for this extended class are closely related to those obtained by Hosgood \cite{Ho2}.

\subsection{Calculations via Maurer Cartan type forms}

We first rewrite the Chern character map $\Ch: \VB \to \OMdR$ in a slightly different notation than the one used in section \ref{SEC:Tot-of-Chern-CW-BT}, but which will be more suitable for comparing it to the formulas from \cite{GMTZ1} and \cite[Section 5]{Ho2}.
\begin{definition}
Let $\vec{\mc E}= (\mc E_0\stackrel {f_1}\longrightarrow \dots \stackrel {f_n} \longrightarrow \mc E_n)\in \VB(M)_n$ be in the nerve, where each $\mc E_j=(E_j\to M,\nabla_j)$ is a complex vector bundle with connection. We again denote by $\vec{\mc E}_\Delta=(E_\Delta=E_0\times |\Delta^n|\to M\times \Delta^n,\nabladelta) \in \BVB(M)_n$ the complex vector bundle as defined in definition \ref{DEF:E-to-EDelta} with the connection $\nabladelta$ given in equation \eqref{EQU:nabdladelta}. For $j=1,\dots, n$, we denote by $\theta_{j}\in \Om^1(M,End(E_0))$  the Maurer Cartan $1$-forms
\begin{equation}\label{EQU:DEF-theta_j}
\theta_{1}:=\nabladelta_{,0}-\nabladelta_{,1}, \quad\quad\theta_{2} := \nabladelta_{,1}-\nabladelta_{,2}, \quad\quad\dots, \quad\quad\theta_{n}:=\nabladelta_{,n-1}-\nabladelta_{,n},
\end{equation}
where $\nabladelta_{,j}$ are the tranferred connections $\nabladelta_{,j}:=f_1^{-1}\circ \dots \circ f_j^{-1}\circ \nabla_j\circ f_j \circ \dots \circ f_1$ on $E_0$ as in \eqref{EQU:nabdladeltaj}. Then,  the connection on $E_\Delta:=E_0\times |\Delta^n|\to M\times |\Delta^n|$ from \eqref{EQU:nabdladelta} can be written as
\begin{equation}\label{EQU:nabla-Delta-via-thetas}
\nabladelta=d_{|\Delta^n|}+\nabladelta_{,n}+t_1\cdot \theta_1+t_2\cdot\theta_2+\dots + t_n\cdot \theta_n.
\end{equation}
Note, that we can rewrite the $1$-forms $\theta_j$ as
\begin{align*}
\theta_j
&=\nabla_{\Delta,j-1}-\nabla_{\Delta,j}=f_1^{-1}\circ \dots \circ f_{j-1}^{-1}\circ\Big(\nabla_{j-1}- f_j^{-1}\circ \nabla_j\circ f_j \Big)\circ f_{j-1} \circ \dots \circ f_1
\\ 
&=f_1^{-1}\circ \dots \circ f_{j-1}^{-1}\circ\Big( f_j^{-1}\circ \nab{Hom(E_{j-1},E_j)}(f_j) \Big)\circ f_{j-1} \circ \dots \circ f_1,
\end{align*}
where $\nab{Hom(E_{j-1},E_j)}(f)=f\circ \nabla_{j-1}-\nabla_j\circ f$ is the induced connection on the $Hom$-bundle. Simplifying notation, we may simply write this as
\begin{equation}\label{EQU:theta_j-from-f_j}
\theta_j=(f_{j-1}\dots f_1)^{-1} \cdot f_j^{-1}\nabla(f_j)\cdot  (f_{j-1}\dots f_1).
\end{equation}

Furthermore, for $j=1,\dots, n$, we denote the curvatures of the connection $\nabla_j$ on $E_j$ by $R_j:=\nabla_j^2\in \Om^2(M,End(E_j))$. Then, the curvature $R_{\Delta,j}\in \Om^2(M,End(E_0))$ of $\nabla_{\Delta,j}$ on $E_0$ is
\begin{equation}
R_{\Delta,j}=\nabla_{\Delta,j}\circ\nabla_{\Delta,j}=f_1^{-1}\circ \dots \circ f_j^{-1}\circ \nabla_j\circ \nabla_j\circ f_j \circ \dots \circ f_1= (f_{j}\dots f_1)^{-1} R_j (f_{j}\dots f_1).
\end{equation}
Note, in particular, that $R_j=0$ iff $R_{\Delta,j}=0$.
\end{definition}

We calculate the curvature of $\nabladelta$ on $E_\Delta=E_0\times |\Delta^n|\to M\times |\Delta^n|$ as follows.
\begin{proposition}\label{PROP:R-from-MC-forms}
The curvature $R_\Delta\in \Om^2(M\times |\Delta^n|, End(E_\Delta))$ of $\nabladelta$ from \eqref{EQU:nabla-Delta-via-thetas} is
\begin{equation}\label{EQU:R(ExDeltan)}
R_\Delta
=\sum_{j=1}^n dt_j \theta_j 
+ \sum_{i,j=1}^n (t_i \cdot t_j-t_{\min(i,j)})\cdot \theta_i\theta_j
+\sum_{j=0}^{n}(t_{j+1}-t_j)\cdot R_{\Delta,j},
\end{equation}
where $\min(i,j)$ denotes the minimum of $i$ and $j$, and, in the third sum on the right, we used the convention that $t_0=0$ and $t_{n+1}=1$.
\end{proposition}
The proof of this proposition is given in section \ref{SEC:R-proof} below. A particular case of interest is when all connections $\nabla_j$ are flat, i.e., when $R_{\Delta,j}=0$ for all $j$. Obviously, in this case, the last sum on the right of \eqref{EQU:R(ExDeltan)} vanishes.

\begin{remark}\label{RMK:expand-CH-as-MC-forms}
Substituting the curvature $R_\Delta$ from \eqref{EQU:R(ExDeltan)} into the expression \eqref{EQU:CH(M)_n(E)(e)} for $\Ch(M)_n(\mc E)$, and evaluating the $p$th powers ($p\geq 0$), we get an expression of the following form:
\begin{align}\label{EQU:CHunE-expanded-first}
\Ch(M)_n(\mc E)(e_{i_0,\dots, i_\ell})& =
\sum_{p\geq 0}  \frac{u^p}{p!\cdot (2\pi i)^p}\cdot
 \int_{|e_{i_0,\dots, i_\ell}|} \tr\Big(\big(R_\Delta\big|_{M\times |e_{i_0,\dots, i_\ell}|}\big)^p\Big)
\\ & = \nonumber
\sum_{p\geq 0}  \frac{u^p}{p!\cdot (2\pi i)^p}\cdot 
\sum_{j_1,\dots, j_p}
C_{n,(i_0,\dots, i_\ell)}^{j_1,\dots, j_p}
\cdot \tr\big(\omega_{1}\cdots \omega_{p}\big)
\end{align}
In more details, the second equality follows from the expansion of the exponential of the curvature $R_\Delta$ and taking the integral of the $\ell$-simplex $|e_{i_0,\dots, i_\ell}|$. Thus, each $\omega_k$ is a choice of $\theta_i$ or $\theta_i\theta_j$ or $R_{\Delta,j}$, appearing in \eqref{EQU:R(ExDeltan)}. The constant appearing in front of a specific term $\tr\big(\omega_{1}\cdots \omega_{p}\big)$ is an integral over an $\ell$-simplex $|e_{i_0,\dots, i_\ell}|$ of a product of $dt_i$ or $(t_i \cdot t_j-t_{\min(i,j)})$ or $(t_{j+1}-t_j)$ which correspond to the coefficients of each $\omega_i$ in \eqref{EQU:R(ExDeltan)}. To keep track which form $\omega_k$ is in the $k$th spot, we labeled these choice by $j_k$, and we note that the constant $C_{n,(i_0,\dots, i_\ell)}^{j_1,\dots, j_p}$ depends on which of these $\omega_k$ were used.

Since $C_{n,(i_0,\dots, i_\ell)}^{j_1,\dots, j_p}$ is an integral of a polynomial (in the $t_i$s) over a simplex, we see that this is in fact a rational number $C_{n,(i_0,\dots, i_\ell)}^{j_1,\dots, j_p}\in \Q$. For example, in the case where all curvatures $R_{\Delta,j}$ vanish, i.e., $R_{\Delta,j}=0$ for all $j$, these coefficients are integrals of the form
\begin{equation}
C_{n,(i_0,\dots, i_\ell)}^{\kappa_1,\dots, \kappa_p}=\pm \int_0^1 \int_0^{t_{\ell}}\dots \int_0^{t_2} t_1^{p_1}(t_1-1)^{q_1}\dots t_\ell^{p_\ell}(t_\ell-1)^{q_\ell} dt_1\dots dt_{\ell-1}dt_\ell
\end{equation}
for some non-negative integers $p_1,\dots,q_\ell$.

Note also that the coefficient $C_{n,(i_0,\dots, i_\ell)}^{j_1,\dots, j_p}$ is non-zero only when the integral over $|e_{i_0,\dots, i_\ell}|$ contains precisely one of each of the forms $dt_{i_1}, \dots dt_{i_\ell}$, which means that the $\omega_k$ include precisely one of the forms $\theta_{i_1},\dots, \theta_{i_\ell}$, since $dt_k$ only appears in the combination $dt_k\theta_k$ in \eqref{EQU:R(ExDeltan)}. This confirms lemma \ref{LEM:degrees-of-Chern-from-local}, that for a simplicial degree $\ell$, the de Rham forms are of degrees $\geq \ell$.

To be more precise, in the exponential expansion in \eqref{EQU:CHunE-expanded-first} for a term with a power of $p$, we have that  $\tr(\omega_{1}\cdots \omega_{p})\in \Om^{2p-\ell}(M,\C)$, since the $\omega_k$ consist of exactly $\ell$ many $1$-forms $\theta_j$ and $p-\ell\geq 0$ many $2$-forms (either $\theta_i\theta_j$ or $R_{\Delta,j}$). From this, we also note that this gives the expected total degree (see definition \ref{DEF:basic-CS-map}) of $||\Ch(M)_n(\mc E)(e_{i_0,\dots, i_\ell})||=-\ell$, since $|u|=-2$.
\end{remark}

Here are some examples for the expansion \eqref{EQU:CH(M)_n(E)(e)} in the above setting.
\begin{example}
On a vertex $e_{i_0}$, \eqref{EQU:CH(M)_n(E)(e)} becomes the Chern character with formal variable $u$ as in \eqref{EQU:ChuE-expanded}, i.e., $\Ch(M)_n(\mc E)(e_{i_0})=\sum_{p\geq 0} \frac{u^p}{p!\cdot (2\pi i)^p}\cdot \tr(R_{i_0}^p)$. In particular, when $R_{i_0}=0$, we get $\Ch(M)_n(\mc E)(e_{i_0})=\tr(Id)=\dim(E)$.
\end{example}

For the next two examples, it will be useful to state a naturality property of the expression for the curvature from \eqref{EQU:R(ExDeltan)} in terms of Maurer Cartan forms.
\begin{lemma}\label{LEM:R-on-ell-subsimplex}
Let $\vec{\mc E}= (\mc E_0\stackrel {f_1}\longrightarrow \dots \stackrel {f_n} \longrightarrow \mc E_n)\in \VB(M)_n$. Restricting the curvature $R_\Delta$ from \eqref{EQU:R(ExDeltan)} on $M\times |e_{0,\dots,n}|$ to $M\times |e_{i_0,\dots,i_\ell}|$ for some $0\leq i_0<\dots< i_\ell\leq n$ is given by the similar expression
\begin{equation}\label{EQU:R(Ex|e-ell|)}
R_{|e_{i_0,\dots,i_\ell}|}
=\sum_{j=1}^\ell d\tilde{t}_j \tilde{\theta}_j
+ \sum_{i,j=1}^\ell (\tilde{t}_i \cdot \tilde{t}_j-\tilde{t}_{\min(i,j)})\cdot \tilde{\theta}_i\tilde{\theta}_j
+\sum_{j=0}^{\ell}(\tilde{t}_{j+1}-\tilde{t}_j)\cdot \tilde{R}_{\Delta,j},
\end{equation}
where $\tilde{\theta}_j=\theta_{i_{j-1}+1}+\dots+\theta_{i_j}$ and $\tilde{R}_{\Delta,j}=R_{\Delta,i_j}$ are the forms and curvatures associated to 
\begin{equation}\label{EQU:E-vec-tilde}
\tilde{\vec{\mc E}}= (\mc E_{i_0}\stackrel {f_{i_1}\circ\dots \circ f_{i_0+1}}\longrightarrow \mc E_{i_1}\quad \longrightarrow\quad \dots \longrightarrow \quad \mc E_{i_{\ell-1}}\stackrel {f_{i_\ell}\circ\dots \circ f_{i_{\ell-1}+1}} \longrightarrow \mc E_{i_\ell}).
\end{equation}
transferred to $E_0$ via the bundle isomorphism $f_{i_0}\circ\dots \circ f_{1}:E_0\to E_{i_0}$.
\end{lemma}

Lemma \ref{LEM:R-on-ell-subsimplex} is proved in section \ref{SEC:R-naturality-proof} below. We next look at the Chern character applied to an edge $e_{i_0,i_1}$.
\begin{example}\label{EXA:Chern-on-edge-e01}
On an edge $e_{i_0,i_1}$, we use the previous lemma to write the curvature on $M\times |e_{i_0,i_1}|$ as the on given by the $1$-simplex $\tilde{\vec{\mc E}}= (\mc E_{i_0}\stackrel {f_{i_1}\circ\dots \circ f_{i_0+1}}\longrightarrow \mc E_{i_1})$. Then, in \eqref{EQU:CHunE-expanded-first} we only get one Maurer Cartan form $\tilde{\theta}_1=\theta_{i_0+1}+\dots+\theta_{i_1}$, and all other $\omega_{k}$ are either $\tilde{\theta}_1^2$, $\tilde{R}_{\Delta,0}$, or $\tilde{R}_{\Delta,1}$. Assuming flatness, i.e., $R_{i_0}=0$ and $R_{i_1}=0$, we obtain
\begin{align*}
\Ch(M)_n(\mc E)(e_{i_0,i_1})
&=\sum_{p\geq 1}  \frac{u^p}{p!\cdot (2\pi i)^p}\cdot \sum_{j_0=1,\dots, p} \int_0^1 (\tilde{t}_1^2-\tilde{t}_1)^{p-1} d\tilde{t}_1 \cdot \tr(\underbrace{\tilde{\theta}^2_1\dots\tilde{\theta}^2_1}_{j_0-1}\tilde{\theta}_1\underbrace{\tilde{\theta}^2_1\dots\tilde{\theta}^2_1}_{p-j_0})\\
&=\sum_{p\geq 1}  \frac{u^p}{p!\cdot (2\pi i)^p}\cdot p\cdot \frac{(-1)^{p-1} (p-1)!(p-1)!}{(2p-1)!} \cdot \tr(\tilde{\theta}_1^{2p-1})\\
&=\sum_{p\geq 1} u^p \frac{(-1)^{p-1}(p-1)!}{(2\pi i)^p (2p-1)! }\cdot \tr(\tilde{\theta}_1^{2p-1})
\end{align*}
Note, that this is just the Chern Simons form associated to the $1$-parameter family of connections $\nabla_t=d_{|\Delta^1|}+\nabla_{\Delta,i_{1}}+t\tilde{\theta}_1$, filtered via the additional factors of $u^p$; see e.g. \cite[Lemma 2.2]{TWZ}.
\end{example}

Our last example concerns the lowest $u$-term of the Chern character for a general $\ell$-simplex.
\begin{example}\label{EXA:lowest-Chern-symmetrized}
For any $\Ch(M)_n(\mc E)(e_{i_0,\dots, i_\ell})$, the lowest non-zero term (in $u$) is the one with $u^\ell$. We use the tilde-notation from lemma \ref{LEM:R-on-ell-subsimplex}. Then, on the right-hand side of equation \eqref{EQU:CHunE-expanded-first}, the integral over the $\ell$-simplex is non-zero only when we take a product of each of the $\ell$ many $2$-forms $d\tilde{t}_j \tilde{\theta}_j$ (from $R_\Delta$ in \eqref{EQU:R(ExDeltan)}) exactly once. The products of these $d\tilde{t}_j \tilde{\theta}_j$ can, however, appear in any order. Thus, it follows that the trace $\tr(\omega_1\dots \omega_\ell)$ in \eqref{EQU:CHunE-expanded-first} contains a product of each $\tilde{\theta}_j$, but in each possible order. Using that the volume of the $\ell$-simplex is $\int_{|\Delta^\ell|}d\tilde{t}_1\dots d\tilde{t}_\ell=\frac{1}{\ell!}$, we therefore obtain
\begin{equation}\label{EQU:CH-U-0...l-lowest-ul}
\Ch(M)_n(\mc E)(e_{i_0,\dots, i_\ell})=\frac{u^\ell}{\ell!(2\pi i)^\ell} \sum_{\sigma\in S_\ell}\frac{\sgn(\sigma)}{\ell!}\tr\big(\tilde{\theta}_{\sigma(1)}\dots \tilde{\theta}_{\sigma(\ell)}\big)+\sum_{p\geq \ell+1}u^{p}\cdot (\dots),
\end{equation}
where $S_\ell$ denotes the set of $\ell$-permutations.
\end{example}

\begin{remark}
Note, that the lowest $u^\ell$-term in \eqref{EQU:CH-U-0...l-lowest-ul} is a symmetrized version of the formula we had in our previous paper \cite[p. 1065 equation (2-1)]{GMTZ1} (with the minor difference of signs and factors of $2\pi i$). An analogous observation is made by Hosgood in \cite[Theorem 5.5.1]{Ho2}. The issue of providing a non-symmetrized version will be taken up in the next proposition \ref{PROP:Un-symmetrizing-Chern}.

Furthermore, we note that in contrast to \cite{GMTZ1}, using a non-vanishing de Rham differential in Definition \ref{DEF:basic-CS-map}, also requires to have non-vanishing higher $u$-terms $u^{\ell+1}, u^{\ell+2},\dots$. In this sense, ``turning on'' the de Rham differential $d$ in \cite{GMTZ1} will necessarily ``generate'' higher $u$-terms in the expansion of the Chern character map.
\end{remark}

There is a modified version of $\Ch$ which has the un-symmetrized terms as its lowest $u$-power.
\begin{proposition}\label{PROP:Un-symmetrizing-Chern}
There is a map of simplicial presheaves $\ChL: \VB \to \OMdR$ such that the lowest $u$-terms are given by:
\begin{equation}\label{EQU:CHL-lowest-ul}
\ChL(M)_n(\mc E)(e_{i_0,\dots, i_\ell})=\frac{u^\ell\cdot \sgn(\ell,\dots,1)}{\ell!(2\pi i)^\ell} \cdot \tr\big(\tilde{\theta}_{\ell}\dots \tilde{\theta}_{1}\big)+\sum_{p\geq \ell+1}u^{p}\cdot (\dots)
\end{equation}

Moreover, \eqref{EQU:CHL-lowest-ul} differs from $\Ch(M)_n(\mc E)(e_{i_0,\dots, i_\ell})$ only in the powers $u^\ell$ and $u^{\ell+1}$. This difference $(\ChL-\Ch)(M)_n(\mc E)\in {\Chain^-}(N(\Z\Delt{n}),\OdR^\bu(M)\ul)$ is a $(d-\sum_j (-1)^j\delta_j)$-exact chain map, which is given by sums of terms involving the de Rham $d$ of the trace of the Maurer Cartan forms $\tilde{\theta}_{1}, \dots, \tilde{\theta}_{\ell}$; see \eqref{EQU:def-of-ChL}, \eqref{EQU:d-on-b-ell}, and \eqref{EQU:delta-on-b-ell}.
\end{proposition}
Proposition \ref{PROP:Un-symmetrizing-Chern} will be proved in section \ref{SEC:unsymmetrized-Ch-proof}. We note that in \eqref{EQU:CHL-lowest-ul} the sign of the permutation $(\ell,\ell-1,\dots, 2,1)$ is given by $\sgn(\ell,\dots,1)=(-1)^{\frac{\ell(\ell-1)}{2}}$.

\subsection{Application to holomorphic bundles with holomorphic connections}
We now show that the above Chern character map defined on complex vector bundles with connections ``factors'' to the holomorphic case. In particular, we obtain an extension of the trace of the exponentiated Atiyah class which is not only \v{C}ech-closed, but ($\partial$+\v{Cech})-closed.

We first establish some notation, and point out differences to the notation from \cite{GMTZ1}. 
\begin{definition}
Denote by $\HVBnabla(M)$ the category of holomorphic vector bundles with holomorphic connection over a complex manifold $M$ and morphisms consist of holomorphic bundle isomorphisms over the identity of $M$, which do not have to respect the connections. Moreover, denote by $\HVB(M)=\Nerve(\HVBnabla(M))$ the nerve of this category; cf. definition \ref{DEF:vb-nabla-VB-nabla} and \cite[Definition 2.2]{GMTZ1} (where we wrote ${\bf HVB}(M)$ instead of $\HVB(M)$).

From the definition of a holomorphic vector bundle $E\to M$ of rank $d$, we can always find an open cover $\mc U=\{U_i\}_{i\in I}$ of $M$ and trivializations $s_i:E_i:=\C^d\times U_i \stackrel\cong\to E|_{U_i}$ whose transition functions $g_{i,j}:=(s_i^{-1}|_{U_{i,j}})\circ (s_j|_{U_{i,j}}):E_j|_{U_{i,j}}\to E|_{U_{i,j}}\to E_i|_{U_{i,j}}$ (satisfying $g_{i,i}=id$ and the cocycle conditions $g_{i,j}\circ g_{j,k}=g_{i,k}$ on $U_{i,j,k}$) are holomorphic functions interpreted as $g_{ij}:U_{i,j}\to GL(d,\C)$, i.e., $\delbar(g_{i,j})=0$. 

A connection on $E$ can be locally given by holomorphic $1$-forms $A_i\in \Om^1_{hol}(U_i,\C^{d\times d})$ for each open set $U_i$ satisfying a compatibility condition $A_i=g_{i,j}\cdot A_j\cdot g_{j,i}+g_{i,j}\cdot \del(g_{j,i})$, and with this data, we have the holomorphic connection $\nabla$ on $E$ given locally by $\nabla|_{U_i}=\del+A_i$. 

We note that, while the $\delbar$ operators on the $U_i$s glue to a global $\delbar$ operator on $E$ (since $\delbar(g_{i,j})=0$ implies $g_{i,j}^{-1}\circ \delbar\circ g_{i,j}=\delbar$), the $\del$ operators on the $U_i$s do, in general, \emph{not} glue to a global $\del$ on $E$ (they only glue in the flat case where $\del(g_{i,j})=0$). In fact, in general, there may not exist a global connection on $E$ (---the usual construction of a global smooth connection using a partition of unity will not give a \emph{holomorphic} connection). We note that the Atiyah class is an obstruction to the existence of a holomorphic connection; see \cite[p. 188, Theorem 2]{A}. For a holomorphic vector bundle $E$, the Atiyah class may be represented via $a(E)=\{ (s_i|_{U_{i,j}})\circ \del(g_{i,j})\circ (s^{-1}_j|_{U_{i,j}})\}_{i,j}\in \vC^1(\mc U,End(E)\otimes \Om^1_{hol})$, since the \v{C}ech $\delta$-cohomology class $[a(E)]=0$ iff $\forall i: \exists B_i\in \Om^1_{hol}(U_{i}, End(E)): B_j-B_i=s_i\circ \del(g_{i,j})\circ s^{-1}_j$ on $U_{i,j}$, which,  setting $A_i=s_i^{-1}\circ B_i\circ s_i$, is equivalent to $s_j\circ A_j\circ s_j^{-1}-s_i\circ A_i\circ s_i^{-1}=s_i\circ \del(g_{i,j})\circ s^{-1}_j$, or $g_{i,j}\cdot A_j\cdot g_{j,i}-A_i= \del(g_{i,j})\cdot g_{j,i}=-g_{i,j}\cdot \del(g_{j,i})$ (using that $g_{i,j}\cdot g_{j,i}=id\implies \del(g_{i,j})\cdot g_{j,i}+g_{i,j}\cdot \del(g_{j,i})=0$ in the last equality).

We also note, that a holomorphic connection is not a special case of a smooth connection, since those are locally given by $d+A_i$ where $d=\del+\delbar$ and where the $A_i\in \OdR^1(U_i,\C^{d\times d})$ satisfy the condition $A_i=g_{i,j}\cdot A_j\cdot g_{j,i}+g_{i,j}\cdot d(g_{i,j})$. However, to a holomorphic connection given by $(\del+A_i)$s, we can associate the smooth connection given by $(\del+\delbar+A_i)$s by adding the $\delbar$ operator (the condition for the $A_i$s from the holomorphic setting implies the one for the smooth setting since $\delbar(g_{i,j})=0$). For a complex manifold $M$, this gives a functor $\HVBnabla(M)\to \vbn(M), (E\to M,\nabla)\mapsto(E\to M, \nabla+\delbar)$, which induces a functor on the nerve, $\HVB(M)\to \VB(M)$. Denoting $\VB$ by slight abuse of notation for the restriction to the site of complex manifolds $\VB:\CMan^{op}\to \sSet$, we get a map of simplicial presheaves $\HVB\to \VB$.
\end{definition}

\begin{definition}\label{DEF:Omega-hol}
For a complex manifold $M$, we denote by $\Omdelhol \bu (M)$ the holomorphic forms on $M$ which we make into a chain complex with the differential $\del$. In contrast to this, we denote by $\Omhol \bu (M)$ the holomorphic forms on $M$ which we make into a chain complex with the zero differential (in \cite[Definition 2.3]{GMTZ1} we simply wrote $\Om^\bu_{hol}(M)$ for $\Omhol \bu (M)$).

Composing this with the functors $\quot$, $\DK$, and $\mc F$ from definition \ref{DEF:OMdR}, we define
\begin{align*}
\OMdelhol & :\CMan^{op}
\stackrel{\Odelhol^\bu(-)}{\longrightarrow}\Chain^+
\stackrel{\quot}{\longrightarrow}\Chain^-
\stackrel{\DK}{\longrightarrow}\sAb
\stackrel{\mc F}{\longrightarrow}\sSet
\\
\OMdelhol(M) & =\DKSet(\Odelhol^\bu(M)\ul)={\Chain^-}(N(\Z\Delt{\bu}),\Odelhol^\bu(M)\ul)
\end{align*}
(In \cite[Definition 2.3]{GMTZ1} we simply wrote $\OMhol$ for the composition $\OMhol :\CMan^{op}\stackrel{\Ohol^\bu(-)}{\longrightarrow}\Chain^+ \stackrel{\quot}{\longrightarrow}\Chain^- \stackrel{\DK}{\longrightarrow}\sAb \stackrel{\mc F}{\longrightarrow}\sSet$.)

Since the inclusion $\Omdelhol \bu (M)\hookrightarrow \OdR ^\bu(M)$ is a chain map (since $\del(\om)=d(\om)$ when $\delbar(\om)=0$), we get an induced map of simplicial presheaves $\OMdelhol\to \OMdR$, where we again used a slight abuse of notation to denote $\OMdR:\CMan^{op}\to \sSet$ for the restriction of definition \ref{DEF:OMdR} to the site of complex manifolds.
\end{definition}

\begin{proposition}\label{PROP:Ch->-factors-to-hol}
The Chern character function factors through holomorphic forms, i.e., there exists a map of simplicial presheaves $\ChLhol$ making the following diagram commute:
\begin{equation}
\xymatrix{ \VB \ar^{\ChL}[rr]  && \OMdR   \\ 
\HVB \ar@{.>}[rr]^{\ChLhol} \ar[u]^{} && \OMdelhol \ar[u]^{}}
\end{equation}
\end{proposition}
\begin{proof}
We need to check that the composition $\HVB\to \VB\stackrel {\ChL}\to\OMdR$ lands in holomorphic forms. However, the output under the Chern character map $\Ch$ in \eqref{EQU:CHunE-expanded-first} only involves traces the Maurer Cartan forms $\theta_j$, and according to  proposition \ref{PROP:Un-symmetrizing-Chern}, for $\ChL$ we have traces and the de Rham $d$ applied to these forms. However, in the holomorphic setting, these Maurer Cartan forms $\theta_j$ are holomorphic, since they are given by holomorphic expressions \eqref{EQU:theta_j-from-f_j}, and so are their traces, etc. ($d(\om)=\del(\om)$ for holomorphic $\om$). Thus, the outcome for $\ChL(M)_n(\mc E)(e_{i_0,\dots, i_\ell})$ in proposition \ref{PROP:Un-symmetrizing-Chern} is also holomorphic, which is what we needed to show.
\end{proof}

We can also apply the totalization from section \ref{SEC:Tot-of-Chern-CW-BT} to the above map $\ChLhol$ for a given open cover $\mc U$. We get an induced map of simplicial sets
\begin{equation}
\Tot(\ChLhol(\NU)): \Tot(\HVB(\NU)) \to \Tot(\OMdelhol(\NU))
\end{equation}

\begin{remark} \label{REM:smooth-del-or-no-del}
Just as in proposition \ref{PROP:Tot(Omega)} and \cite[Prop. 3.13]{GMTZ1}, the totalization gives a map $\Tot(\OMdelhol(\NU))\to \DKSet(\vC^\bu(\mc U,\Odelhol^\bu)[u]^{\leq 0})$. For simplicial degree $0$ we get  the output $\DKSet(\vC^\bu(\mc U,\Odelhol^\bu)[u]^{\leq 0})_0\cong \vC^\bu(\mc U,\Odelhol^\bu)^{even}$, while two $0$-simplices connected by a $1$-simplex give forms whose difference is exact. Here, the corresponding differential is the \v{C}ech-$\del$ differential, since we included the $\del$-differential in the holomorphic forms $\Odelhol^\bu(M)$ in definition \ref{DEF:Omega-hol}. The output from the Chern character can thus be understood as an element in $H^\bu(\vC^\bu(\mc U,\Odelhol^\bu),\delta+\del)$, which is in fact isomorphic to de Rham cohomology $H^\bu_{\text{dR}}(M,\C)$; see \cite[p. 448]{GH}.

If, however, we would have taken holomorphic forms $\Omhol \bu (M)$ with zero differential, we would have ended in the complex $\vC^\bu(\mc U,\Omhol \bu)$ with $\delta+0=\delta$ differential. This is in fact the setup we studied in \cite{GMTZ1} using only the ``diagonal'' terms of the Chern character map, i.e., using only terms with equal \v{C}ech and holomorphic degree. Unlike the smooth case, where a zero differential does not yield an interesting cohomology (see remark \ref{REM:smooth-d-or-no-d}), the cohomology $H^\bu(\vC^\bu(\mc U,\Omhol \bu),\delta)$ is Hodge cohomology (\cite{OTT}), and which is isomorphic to Dolbeault cohomology $H^{\bu,\bu}_{\delbar}(M)$ via the Dolbeault theorem (see e.g. \cite[p. 45]{GH}).
\end{remark}

\begin{remark}\label{REM:lowest-terms-extended}
Another benefit of applying the totalization is that it allows us to write our starting data in terms of local data of the manifold (as e.g. stated in proposition \ref{PROP:Tot(VB)01}). This was done for example in  the smooth case in \ref{EXA:Bott-Tu-example}, where we recovered the formulas by Bott and Tu. In the holomorphic case we recover an extension of the formulas of the Chern character map from \cite{GMTZ1}, but now with the $\del$-differential ``turned on''. More precisely, for a cover $\mc U$ of $M$, applying $\Tot(\ChLhol(\NU))$ to a $0$-simplex given by a holomorphic vector bundle $E\to M$ and a choice of holomorphic connections $\nabla_i$ on each $E_i=E|_{U_i}\to U_i$ (just as in proposition \ref{PROP:Tot(VB)01}) the sequence of bundles $((E_{i_0},\nabla_{i_0})\stackrel {g_{i_1,i_0}}\longrightarrow \dots \stackrel {g_{i_n,i_{n-1}}} \longrightarrow (E_{i_n},\nabla_{i_n}))$ over $U_{i_0,\dots, i_n}$ yields the $1$-forms $\tilde{\theta}_j=g_{i_{0},i_j}\circ \nabla(g_{i_j,i_{j-1}})\circ g_{i_{j-1},i_0}$ from \eqref{EQU:theta_j-from-f_j}, and with this \eqref{EQU:CHL-lowest-ul} becomes an extension of the trace of the exponentiated Atiyah class which is $(\delta+\del)$-closed:
\begin{equation}\label{EQU:Atiyah-extended-with-del}
 \frac{ u^n}{n!(2\pi i)^n} \cdot (-1)^{\frac{n(n-1)}{2}}\cdot\tr\big(g_{i_0,i_n}\circ \nabla(g_{i_n,i_{n-1}})\circ\dots \circ \nabla(g_{i_1,i_{0}})\big)+\sum_{p\geq n+1} u^{p}\cdot (\dots)
\end{equation}

In fact, one of our goals is to obtain a similar description of the Chern character for coherent sheaves as studied by O'Brian Toledo and Tong in \cite{OTT} expressed in terms of local data. We have already done this in case without $\del$ in \cite{GMTZ2}, and this paper will hopefully give a step toward giving a full description, which allows us to ``turn on the $\del$'' for infinity vector bundles as well.

Finally, we mention that the application of the totalization automatically gives invariants of higher structures, such as morphisms, and higher morphisms as well. For example, a map of bundles gives a $1$-simplex in the totalization (cf. proposition \ref{PROP:Tot(VB)01} in the setup from example \ref{EXA:Chern-Weil-example}), which yields a Chern-Simon type invariant (cf. example \ref{EXA:Chern-on-edge-e01}). However, there are also higher simplicies in the totalization, which correspond to certain higher structures for a bundle.
\end{remark}

\begin{example}\label{EXA:Kahler-E1=Einfinity}
Frölicher showed in \cite{F} that there is a spectral sequence whose $E_1$-page is isomorphic to Dolbeault cohomology, and whose $E_\infty$-page is isomorphic to de Rham cohomology. If $M$ is a compact K\"ahler manifold, then \cite{F} showed that this spectral sequence degenerates at the $E_1$-page (see also \cite[p. 444]{GH}).

Similarly, taking the double complex $\vC^\bu(\mc U,\Odelhol^\bu)$ with $\del$ and $\delta$ differentials with appropriate filtration (by form-degree), gives a spectral sequence whose $E_1$-page is isomorphic to Hodge cohomology (which is isomorphic to Dolbeault cohomology), and whose $E_\infty$-page is isomorphic to de Rham cohomology; see \cite[p. 448]{GH}. Thus, when $M$ is compact and K\"ahler, the spectral sequence degenerates at the $E_1$-page. In particular, the $u^n$-terms (for $n\geq 0$) in the expansion \eqref{EQU:Atiyah-extended-with-del} are $\delta$-closed, which are a representative for a class in Hodge cohomology (akin to the class we used in \cite[(2-1)]{GMTZ1}). These $u^n$-terms are extended in the full sum \eqref{EQU:Atiyah-extended-with-del} to give a $(\del+\delta)$-closed representative for the Chern character in de Rham cohomology. It follows that, in the compact K\"ahler case, the identification of Hodge and de Rham cohomology also identifies these classes with each other.\end{example}

\subsection{Proofs of proposition \ref{PROP:R-from-MC-forms}, lemma \ref{LEM:R-on-ell-subsimplex}, and proposition \ref{PROP:Un-symmetrizing-Chern}}\label{SEC:Proof-of-PROPs}

We complete this section by providing the missing proofs for the stated propositions and lemma.

\subsubsection{Proof of Proposition \ref{PROP:R-from-MC-forms}}\label{SEC:R-proof}
\begin{proof}
Since $d_{|\Delta^n|}^2=0$ and $d_{|\Delta^n|}\circ \nabla_{\Delta,n}+\nabla_{\Delta,n} \circ d_{|\Delta^n|}=0$, we get
\begin{multline}\label{EQU:R{ExDelta}-first-eval}
R_\Delta
=(\nabla_\Delta)^2
=\Big(d_{|\Delta^n|}+\nabla_{\Delta,n}+\sum_{j=1}^n t_j\cdot \theta_j\Big)^2
\\
=\Big(d_{|\Delta^n|}\circ \sum_{j=1}^n t_j \theta_j+\sum_{j=1}^n t_j \theta_j \circ d_{|\Delta^n|}\Big)+\nabla_{\Delta,n}^2+\Big(\nabla_{\Delta,n}\circ \sum_{j=1}^n t_j \theta_j+\sum_{j=1}^n t_j \theta_j \circ \nabla_{\Delta,n}\Big)+\Big(\sum_{j=1}^n t_j\cdot \theta_j\Big)^2
\\
=\sum_{j=1}^n dt_j \theta_j+R_{\Delta,n}+\sum_{j=1}^n t_j\cdot\Big(\nabla_{\Delta,n}\theta_j+\theta_j\nabla_{\Delta,n}\Big)+\sum_{i,j=1}^n t_i t_j \theta_i \theta_j.
\end{multline}
We next calculate $\sum_{i,j=1}^n t_{\min}\theta_i \theta_j$, where $\min=\min(i,j)$.
\begin{align*}
&\sum_{i,j=1}^n t_{\min}\theta_i \theta_j 
=\sum_{j=1}^n t_j\cdot \Big(\theta_j\theta_j+\sum_{i=j+1}^n \theta_i\theta_j+\sum_{i=j+1}^n \theta_j\theta_i\Big)
\\
&
=\sum_{j=1}^n t_j\cdot \Big((\nabla_{\Delta,j-1}-\nabla_{\Delta,j})(\nabla_{\Delta,j-1}-\nabla_{\Delta,j})+\sum_{i=j+1}^n (\nabla_{\Delta,i-1}-\nabla_{\Delta,i})\theta_j+\sum_{i=j+1}^n \theta_j(\nabla_{\Delta,i-1}-\nabla_{\Delta,i})\Big)
\\
&
=\sum_{j=1}^n t_j\cdot \Big(R_{\Delta,j-1}-\nabla_{\Delta,j-1}\nabla_{\Delta,j}-\nabla_{\Delta,j}\nabla_{\Delta,j-1}+R_{\Delta,j}+
\\
& \hspace{2cm}
\underbrace{\nabla_{\Delta,j}\theta_j}_{=\nabla_{\Delta,j}\nabla_{\Delta,j-1}-\nabla_{\Delta,j}\nabla_{\Delta,j}}-\nabla_{\Delta,n}\theta_j+\underbrace{\theta_j\nabla_{\Delta,j}}_{=\nabla_{\Delta,j-1}\nabla_{\Delta,j}-\nabla_{\Delta,j}\nabla_{\Delta,j}}-\theta_j\nabla_{\Delta,n} \Big)
\\
&
=\sum_{j=1}^n t_j\cdot \Big(R_{\Delta,j-1}+R_{\Delta,j}-R_{\Delta,j}-\nabla_{\Delta,n}\theta_j-R_{\Delta,j}-\theta_j\nabla_{\Delta,n} \Big)
\\
&
=\sum_{j=1}^n t_j\cdot (R_{\Delta,j-1}-R_{\Delta,j})-\sum_{j=1}^n t_j\cdot (\nabla_{\Delta,n}\theta_j+\theta_j\nabla_{\Delta,n} ).
\end{align*}
Thus, rearranging the sides, and using the notation $t_0=0$ and $t_{n+1}=1$, we get 
\begin{eqnarray}\label{EQU:RDelta-n-as-minij}
&\hspace{6mm} R_{\Delta,n}+\sum_{j=1}^n t_j\cdot (\nabla_{\Delta,n}\theta_j+\theta_j\nabla_{\Delta,n} )
& =R_{\Delta,n}+\sum_{j=1}^n t_j\cdot (R_{\Delta,j-1}-R_{\Delta,j})-\sum_{i,j=1}^n t_{\min}\theta_i \theta_j
\\
\nonumber && =\sum_{j=0}^{n}(t_{j+1}-t_j)\cdot R_{\Delta,j}-\sum_{i,j=1}^n t_{\min}\theta_i \theta_j.
\end{eqnarray}
Plugging \eqref{EQU:RDelta-n-as-minij} into equation \eqref{EQU:R{ExDelta}-first-eval}, we obtain the claimed result \eqref{EQU:R(ExDeltan)}.
\end{proof}

\subsubsection{Prove of lemma \ref{LEM:R-on-ell-subsimplex}}\label{SEC:R-naturality-proof}
\begin{proof}
We need to pull back the curvature under the inclusion map $|e_{i_0,\dots,i_\ell}|\hookrightarrow |e_{0,\dots,n}|$, $(\tilde{t}_1,\dots,\tilde{t}_\ell)\mapsto (t_1,\dots,t_n)=(\underbrace{0,\dots,0}_{i_0\text{ many}},\underbrace{\tilde{t}_1,\dots,\tilde{t}_1}_{i_1-i_0\text{ many}},\underbrace{\tilde{t}_2,\dots,\tilde{t}_2}_{i_2-i_1\text{ many}},\dots,\underbrace{\tilde{t}_\ell,\dots,\tilde{t}_\ell}_{i_\ell-i_{\ell-1}\text{ many}},\underbrace{1,\dots,1}_{n-i_\ell\text{ many}})$, which means that $t_{i_{j-1}+1}=\dots=t_{i_j}=\tilde{t}_j$; compare notation \ref{NOTATION:simplicial-sets}. Then, the terms in the first sum of \eqref{EQU:R(ExDeltan)} are $dt_{i_{j-1}+1}\theta_{i_{j-1}+1}+\dots+dt_{i_{j}}\theta_{i_{j}}=d\tilde{t}_j(\theta_{i_{j-1}+1}+\dots+\theta_{i_{j}})=d\tilde{t}_j\tilde{\theta}_j$, which give the stated terms in the first sum of \eqref{EQU:R(Ex|e-ell|)}. For the middle sum of  \eqref{EQU:R(ExDeltan)}, note that for, say $i<j$, $t_it_j-t_i$ vanishes when $t_i=0$ or $t_j=1$. For the remaining terms, note that
\[
\sum_{i=i_{r-1}+1}^{i_r}\sum_{j=i_{s-1}+1}^{i_s} ({t}_i \cdot {t}_j-{t}_{\min(i,j)})\cdot {\theta}_i{\theta}_j
=
(\tilde{t}_r \cdot \tilde{t}_s-\tilde{t}_{\min(r,s)})\cdot \sum_{i=i_{r-1}+1}^{i_r}\sum_{j=i_{s-1}+1}^{i_s}  {\theta}_i{\theta}_j
\]
which gives the terms $(\tilde{t}_r \cdot \tilde{t}_s-\tilde{t}_{\min(r,s)})\cdot \tilde{\theta}_r\tilde{\theta}_s$ in the second sum of \eqref{EQU:R(Ex|e-ell|)}. For the last sum of  \eqref{EQU:R(ExDeltan)}, we note that ${t}_{j+1}-{t}_j$ vanishes for all $j\neq i_k$; in the case where $j\neq i_k$ for some $k$, we get $({t}_{i_k+1}-{t}_{i_k})\cdot R_{\Delta,i_k}=(\tilde{t}_{k+1}-\tilde{t}_{k})\cdot \tilde{R}_{\Delta,k}$.

Finally, we note that $\theta_{i_{j-1}+1}+\dots+\theta_{i_j}$ can be rewritten using \eqref{EQU:theta_j-from-f_j} and the chain rule $\nabla(g\circ h)=\nabla(g)\circ h+g\circ \nabla(h)$ as a pullback of the Maurer Cartan forms coming from \eqref{EQU:E-vec-tilde}.
\end{proof}

\subsubsection{Proof of proposition \ref{PROP:Un-symmetrizing-Chern}}\label{SEC:unsymmetrized-Ch-proof}

The remainder of this section concerns the proof of proposition \ref{PROP:Un-symmetrizing-Chern}, that is, we show how to modify the Chern character map so that the lowest $u^\ell$-component (described in \eqref{EQU:CH-U-0...l-lowest-ul}) becomes ``non-symmetrized'', i.e., so that we will only have a trace of $\tr\big(\theta_{\ell}\dots \theta_{1})$ instead of a sum.  The idea for being able to rearrange the order of the $\theta_i$s is based on the following lemma.
\begin{lemma}\label{LEM:sigma-permut-for-b}
Let $\sigma\in S_{\ell+1}$ be an $(\ell+1)$-permutation, and let $1\leq j\leq \ell$. Denote by $\sigma_j:\{1,\dots, \ell+1\}\to \{1, \dots, \ell\}$, $\sigma_j(k)=\sigma(k)$ if $\sigma(k)\leq j$, and $\sigma_j(k)=\sigma(k)-1$ if $\sigma(k)> j$. For example, for $\sigma=(4,3,1,5,2)\in S_5$, and $j=2$, we have $\sigma_2=(3,2,1,4,2)$.

Now, for any endomorphism-valued $1$-forms $\phi_i$ for $i=1,\dots, \ell$, we denote by
\[
b_{\sigma_j}(\phi_1,\dots, \phi_{\ell}):=\phi_{\sigma_j(1)}\cdot \ldots\cdot \phi_{\sigma_j(\ell+1)}.
\]
Then, for any endomorphism-valued $1$-forms $\theta_i$, where $i=1,\dots, \ell+1$, we have:
\begin{multline}\label{EQU:Cech-diff-for-one-transpos}
b_{\sigma_j}(\theta_{2},\dots, \theta_{\ell+1})+
\sum_{i=1}^{\ell} (-1)^i\cdot b_{\sigma_j}(\theta_{1},\dots, \theta_{i}+\theta_{i+1},\dots, \theta_{\ell+1})
+(-1)^{\ell+1}\cdot b_{\sigma_j}(\theta_{1},\dots, \theta_{\ell})
\\  
 =(-1)^j\cdot \Big((\theta_{\sigma(1)}\dots \theta_{j}\dots \theta_{j+1}\dots \theta_{\sigma(\ell+1)})+(
\theta_{\sigma(1)}\dots \theta_{j+1}\dots \theta_{j}\dots \theta_{\sigma(\ell+1)})\Big),
\end{multline}
where on the right-hand side only the indices $j$ and $j+1$ were swapped and all other $\sigma(k)$ (with $\sigma(k)\neq j, j+1$) stayed the same.
\end{lemma}
\begin{proof}
We will demonstrate the proof in the particular example stated in the lemma, i.e., for $\sigma=(4,3,1,5,2)$ and $j=2$ giving $\sigma_2=(3,2,1,4,2)$. The general proof works the same way, but the notation required for the general case obscures the simple idea of its proof.

For the given example of $\sigma$ and $j$, the left-hand side of \eqref{EQU:Cech-diff-for-one-transpos} is
\begin{multline*}
 \theta_4\theta_3\theta_2\theta_5\theta_3
-\Big(\theta_4\theta_3(\theta_1+\theta_2)\theta_5\theta_3\Big)
+\Big(\theta_4(\theta_2+\theta_3)\theta_1\theta_5(\theta_2+\theta_3)\Big) \\
-\Big((\theta_3+\theta_4)\theta_2\theta_1\theta_5\theta_2\Big) 
+\Big(\theta_3\theta_2\theta_1(\theta_4+\theta_5)\theta_2\Big) 
-\theta_3\theta_2\theta_1\theta_4\theta_2 
\,\,\,   =\,\,\, \theta_4\theta_3\theta_1\theta_5\theta_2
+\theta_4\theta_2\theta_1\theta_5\theta_3,
\end{multline*}
where consecutive terms always cancel except for the two stated on the right-hand side.

Note, that a similar computation works for a general $\sigma$ and $j$ as well.
\end{proof}

We can use the $b$-function from the previous lemma to build a simplicial presheaf as follows.

\begin{definition}
Fix a set map $\nu:\{1,\dots, \ell+1\}\to \{1, \dots, \ell\}$, we denote by 
\[
\bbb_{\nu}(\phi_1,\dots,\phi_\ell):=u^{\ell+1}\cdot \tr\Big(\phi_{\nu(1)}\cdot \ldots\cdot \phi_{\nu(\ell+1)}\Big).
\]

Using this, we define a map of simplicial presheaves ${\bf{b}_{\nu}}: \VB \to \OMdR$ as follows. Given a smooth manifold $M$, and any $\vec{\mc E}= (\mc E_0\stackrel {f_1}\longrightarrow \dots \stackrel {f_n} \longrightarrow \mc E_n)\in \VB(M)_n$ in the nerve of vector bundles with connections, we need to associate a chain map ${\bf{b}_{\nu}}(M)_n(\vec{\mc E}):N(\Z\Delt{n})\to\OdR^\bu(M)\ul$ (similar to the construction for $\BCh$ in definition \ref{DEF:basic-CS-map}). For any $k$-subcell $|e_{i_0,\dots, i_k}|\subseteq |\Delta^n|$, denote by $\tilde\theta_1,\dots,\tilde\theta_k$ the Maurer Cartan forms associated to $(\mc E_{i_0}\to \mc E_{i_1}\to \dots \to \mc E_{i_k})$ as was done in lemma \ref{LEM:R-on-ell-subsimplex}; i.e. $\tilde\theta_j=\theta_{i_{j-1}+1}+\dots+\theta_{i_j}$. The map ${\bf{b}_{\nu}}(M)_n(\vec{\mc E})$ will be non-zero only on $\ell$ and $(\ell+1)$-simplicies:

$\bu$ Then, for any $\ell$-subcell $|e_{i_0,\dots, i_{\ell}}|\subseteq |\Delta^n|$, we define ${\bf{b}_{\nu}}(M)_n(\vec{\mc E})(e_{i_0,\dots, i_\ell})$ to be the expression obtained from $\bbb_{\nu}(\tilde\theta_1,\dots,\tilde\theta_\ell)$ under the de Rham differential in $\OdR^\bu(M)\ul$:
\begin{equation}\label{EQU:d-on-b-ell}
{\bf{b}}_\nu(M)_n(\vec{\mc E})(e_{i_0,\dots, i_\ell}):=d\Big(\bbb_{\nu}(\tilde\theta_1,\dots,\tilde\theta_\ell)\Big)
=u^{\ell+1}\cdot d\Big(\tr(\tilde\theta_{\nu(1)}\dots\tilde\theta_{\nu(\ell+1)})\Big)
\end{equation}

$\bu$ For any $(\ell+1)$-subcell $|e_{i_0,\dots, i_{\ell+1}}|\subseteq |\Delta^n|$, we define ${\bf{b}_{\nu}}(M)_n(\vec{\mc E})(e_{i_0,\dots, i_{\ell+1}})$ to be the expression obtained from $\bbb_{\nu}(\tilde\theta_1,\dots,\tilde\theta_\ell)$ by taking the differential in $N(\Z\Delt{n})$:
\begin{multline}\label{EQU:delta-on-b-ell}
{\bf{b}_{\nu}}(M)_n(\vec{\mc E})(e_{i_0,\dots, i_{\ell+1}})
\\
:=\bbb_\nu(\tilde\theta_{2},\dots, \tilde\theta_{\ell+1})
+\sum_{i=1}^{\ell} (-1)^i\cdot \bbb_\nu(\tilde\theta_{1},\dots, \tilde\theta_{i}+\tilde\theta_{i+1},\dots, \tilde\theta_{\ell+1}) 
+(-1)^{\ell+1}\cdot \bbb_\nu(\tilde\theta_{1},\dots, \tilde\theta_{\ell})
\end{multline}

$\bu$ All other ${\bf{b}_{\nu}}(M)_n(\vec{\mc E})(e_{i_0,\dots, i_k})$ with $k\neq \ell, \ell+1$ are defined to be zero. 
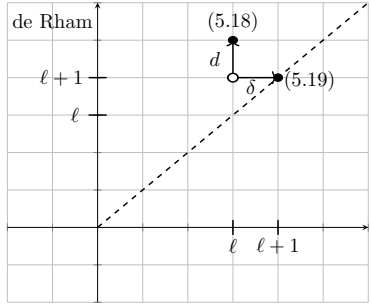
\[
\scalebox{0.7}{
\begin{tikzpicture}
\begin{axis}[axis lines = center, grid=both, xmin=-2, xmax=6, ymin=-2, ymax=6, xtick={-2,...,6},  ytick={-2,...,6}, xticklabels={}, yticklabels={}];
\draw [fill, black] (4,4) circle [radius=.1]; 
\draw [fill, black] (3,5) circle [radius=.1]; 
\draw [dashed, thick] (0,0)--(6,6);
\draw [thick] (3,.2)--(3,-.2); \draw [thick] (4,.2)--(4,-.2);
\draw [thick] (.2,3)--(-.2,3); \draw [thick] (.2,4)--(-.2,4);
\node at (3,-.5) {$\ell$}; \node at (4,-.5) {$\ell+1$};
\node at (-.5,3) {$\ell$}; \node at (-.8,4) {$\ell+1$};
\draw [thick, ->] (3,4)--(4,4); \node at (3.4,3.75) {$\delta$}; 
\draw [thick, ->] (3,4)--(3,5); \node at (2.6,4.5) {$d$};  \node at (-1,5.5) {de Rham};
\node at (3,5.5) {\eqref{EQU:d-on-b-ell}};
\node at (4.7,3.9) {\eqref{EQU:delta-on-b-ell}};
\draw [fill, black] (3,4) circle [radius=.12]; \draw [fill, white] (3,4) circle [radius=.08]; 
\end{axis}
\end{tikzpicture}}
\]
 
{\bf Claim:} ${\bf b_{\nu}}$ is a map of simplicial presheaves.
\begin{proof}[Proof of the Claim.]
We need to check three properties, first that ${\bf{b}}(M)_n(\vec{\mc E})$ is a chain map, second that ${\bf{b}_{\nu}}(M)$ is a map of simplicial sets, and third that $\bf{b}_{\nu}$ is a natural transformation.

First, ${\bf{b}_{\nu}}(M)_n(\vec{\mc E}):N(\Z\Delt{n})\to\OdR^\bu(M)\ul$ is a chain map, since ${\bf{b}_{\nu}}(M)_n(\vec{\mc E})(d(e_{i_0,\dots,i_{\ell+1}}))=\sum_{j=0}^{\ell+1} (-1)^j \cdot {\bf{b}_{\nu}}(M)_n(\vec{\mc E})(e_{i_0,\dots,\widehat{i_j},\dots,i_{\ell+1}})$, which, using \eqref{EQU:d-on-b-ell}, can be evaluated to be:
\begin{itemize}
\item
for $j=0$ on the right in the above sum, we get $d(\bbb_\nu(\tilde\theta_{2},\dots, \tilde\theta_{\ell+1}))$
\item
for $j=1,\dots,\ell$, we get $(-1)^j\cdot d(\bbb_\nu(\tilde\theta_{1},\dots, \tilde\theta_{i}+\tilde\theta_{i+1},\dots,\tilde\theta_{\ell+1}))$, since the Maurer Cartan form for $\mc E_{i_{j-1}}\to \mc E_{i_{j+1}}$ is the sum of the forms $\tilde\theta_{i}+\tilde\theta_{i+1}$ for $\mc E_{i_{j-1}}\to \mc E_{i_{j}}$ and $\mc E_{i_{j}}\to \mc E_{i_{j+1}}$ 
\item
for $j=\ell+1$, we get $(-1)^{\ell+1}\cdot d(\bbb_\nu(\tilde\theta_{1},\dots, \tilde\theta_{\ell}))$
\end{itemize}
From \eqref{EQU:delta-on-b-ell}, we see that this is $d({\bf{b}_{\nu}}(M)_n(\vec{\mc E})(e_{i_0,\dots, i_{\ell+1}}))$. 

Next, ${\bf{b}_{\nu}}(M)$ is a map of simplicial sets, i.e., for a morphism $\rho:[n]\to [m]$, the following diagram commutes (using the notation from the proof of proposition \ref{PROP:Ch-presheaf-map}):
\begin{equation*}
\xymatrix{ \VB(M)_m \ar^{{\bf{b}_{\nu}}(M)_m\hspace{.7in}}[rr] \ar_{\VB(M)_\rho}[d] && {\Chain^-}(N(\Z\Delt{m}),\OdR^\bu(M)\ul) \ar^{\kappa\mapsto \kappa\circ \rho_*}[d]   \\ 
\VB(M)_n \ar^{{\bf{b}_{\nu}}(M)_n\hspace{.7in}}[rr]&& {\Chain^-}(N(\Z\Delt{n}),\OdR^\bu(M)\ul) }
\end{equation*}
Now, $({\bf{b}_{\nu}}(M)_{n}\circ \VB(M)_{\rho}(\vec{\mc E}))(e_{i_0,\dots,i_\ell})=({\bf{b}_{\nu}}(M)_{n}((id\times \rho)^*\vec{\mc E}))(e_{i_0,\dots,i_\ell})$, and we check the cases of face maps and degeneracies.
\begin{itemize}
\item
For a face map $\delta_j:[n]\to [n+1]$ with $n\geq \ell$, we get that $(id\times \delta_j)^*(\mc E_0\to\dots\to \mc E_{n+1})=(\mc E_0\to \dots\to\mc E_{j-1}\stackrel{f_{j+1}\circ f_j}\longrightarrow\mc E_{j+1}\to\dots\to \mc E_{n+1})$, unless $j=0$ or $j=n+1$ in which case the first or last arrow gets removed, respectively. Thus, for $0\leq i_0<\dots<i_\ell\leq n$, the induced sequences of bundles when pulled back to $(\mc E_{i_0}\to\dots\mc \to \mc E_{i_\ell})$ (see \eqref{EQU:E-vec-tilde}) is equal to the pull back of $(\mc E_0\to\dots\to \mc E_{n+1})$ to the corresponding $e_{i'_0,\dots,i'_\ell}=(\delta_j)_*(e_{i_0,\dots,i_\ell})\in N(\Z\Delt{n+1})$ for $j>0$, and they are isomorphic via the extra pull back under $f_1$ in the case $j=0$ (---compare the proof of lemma \ref{LEM:VB-to-BVB}). In either case, taking a trace over Maurer Cartan forms in ${\bbb}_\nu$ gives same de Rham form, $({\bf{b}_{\nu}}(M)_{n}((id\times \delta_j)^*\vec{\mc E}))(e_{i_0,\dots,i_\ell})=({\bf{b}_{\nu}}(M)_{n+1}(\vec{\mc E}))\circ (\delta_j)_*( e_{i_0,\dots,i_\ell})$.
\item
For a degeneracy $\sigma_j:[\ell]\to [\ell-1]$, $(\sigma_j)_*(e_{i_0,\dots,i_\ell})=0$ in the normalized complex $N(\Z\Delt{\ell-1})$ (cf. proposition \ref{PROP:Ch-presheaf-map}). On the other hand $\sigma_j^*(\vec{\mc E})$ has an identity in one of its maps (see lemma \ref{LEM:VB-to-BVB}), which by equation \eqref{EQU:theta_j-from-f_j} will give a Maurer Cartan form of $0$, and so $({\bf{b}_{\nu}}(M)_{\ell}((id\times \sigma_j)^*\vec{\mc E}))(e_{i_0,\dots,i_\ell})=0=({\bf{b}_{\nu}}(M)_{\ell-1}(\vec{\mc E}))\circ (\sigma_j)_*( e_{i_0,\dots,i_\ell})$.
\end{itemize}
The analysis then follows just as in the proof of proposition \ref{PROP:Ch-presheaf-map}, depending on whether $\rho_*(e_{i_0,\dots,i_\ell})$ is degenerate or not.

Finally, we show that for a map of smooth manifolds $g:M'\to M$, the diagram
\begin{equation*}
\xymatrix{ \VB(M) \ar^{\VB(g)}[rr] \ar_{{\bf{b}_{\nu}}(M)}[d] && \VB(M') \ar^{{\bf{b}_{\nu}}(M')}[d]   \\ 
\OMdR(M) \ar^{\OMdR(g)}[rr]&& \OMdR(M') }
\end{equation*}
commutes. Now, for $\vec{\mc E}=(\mc E_0\to\dots\to\mc E_n)\in \VB(M)_n$ over $M$, $ \VB(g)_n (\vec{\mc E})=g^*(\vec{\mc E})=(g^*(\mc E_0)\to \dots\to g^*(\mc E_n))$, so that the induced Maurer Cartan forms \eqref{EQU:theta_j-from-f_j} are the pullbacks of the Maurer Cartan forms from $\vec{\mc E}$. Since the formulas for ${\bf{b}_{\nu}}(M)_n(\vec{\mc E})(e_{i_0,\dots, i_{k}})$ (see \eqref{EQU:d-on-b-ell} and \eqref{EQU:delta-on-b-ell}) are given in terms of these forms, we see that $({\bf{b}_{\nu}}(M')_n\circ \VB(g)_n (\vec{\mc E}))(e_{i_0,\dots, i_k})=g^*({\bf{b}_{\nu}}(M)_n(\vec{\mc E})(e_{i_0,\dots, i_k}))$. This, by definition, is equal to $(\OMdR(g)_n\circ {\bf{b}_{\nu}}(M)_n(\vec{\mc E}))(e_{i_0,\dots, i_k})$, which shows that ${\bf{b}_{\nu}}$ is a natural transformation.
\end{proof}
\end{definition}

Combining the last definition and lemma, we now define a variation of the Chern character map, that gives a non-symmetrized terms on the diagonal.
\begin{definition}\label{DEF:ChL}
The terms of the map $\Ch: \VB \to \OMdR$ in lowest component are sums of over all permutations of products of Maurer Cartan forms; see \eqref{EQU:CH-U-0...l-lowest-ul}. Let $\sigma\in S_{\ell+1}$. Then, there is a sequence of permutations $\sigma^{(0)}=\sigma,\dots,\sigma^{(p)}=(\ell+1,\ell,\dots, 2, 1)$ with $p=p(\sigma)$, where two consecutive permutations differ from each other by a single transposition of outputs, for example, $\sigma^{(k-1)}=(4,3,1,5,2)$ and $\sigma^{(k)}=(4,2,1,5,3)$ having $2$ and $3$ transposed. The position of the transposition is denoted by $j=j(\sigma,k)$; in the example in the previous sentence $j=2$.

We now fix such a sequence of $\sigma^{(k)}$s and transpositions $j(\sigma,k)$ for each permutation $\sigma$; to be concrete, we can start from $\sigma^{(0)}=\sigma$ by first moving $1$ to the last position, then $2$ to the second to last position, and follow in this fashion until we get $\sigma^{(p)}=(\ell+1,\ell,\dots, 2, 1)$. From lemma \ref{LEM:sigma-permut-for-b}, we know that $\sigma^{(k-1)}_{j(\sigma,k)}:\{1,\dots,\ell+1\}\to \{1,\dots,\ell\}$ can be used to interpolate beteeen $\sigma^{(k-1)}$ and $\sigma^{(k)}$, as stated in \eqref{EQU:Cech-diff-for-one-transpos}. We therefore define $\ChL: \VB \to \OMdR$ to be
\begin{equation}\label{EQU:def-of-ChL}
\ChL:=\Ch+{\bf{b}},\text{ with }{\bf{b}}:=\sum_{\ell\geq 1} \sum_{\sigma\in S_{\ell+1}} \sum_{k=1}^{p(\sigma)} \frac{\sgn(\sigma^{(k-1)})\cdot (-1)^{j(\sigma,k)+1}}{(\ell+1)!(2\pi i)^{\ell+1}}\cdot {{\bf{b}}_{\sigma^{(k-1)}_{j(\sigma,k)}}}
\end{equation}
We note that the additional ${\bf b}_\nu$s are non-vanishing only on the diagonal and off-diagonal degrees for $\ell\geq 1$:
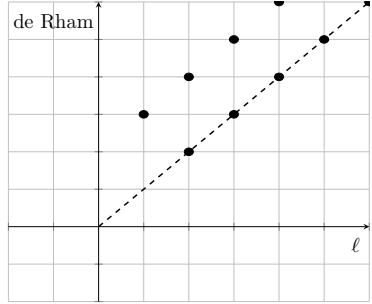
\[
\scalebox{0.7}{
\begin{tikzpicture}
\begin{axis}[axis lines = center, grid=both, xmin=-2, xmax=6, ymin=-2, ymax=6, xtick={-2,...,6},  ytick={-2,...,6}, xticklabels={}, yticklabels={}];
\draw [fill, black] (1,3) circle [radius=.1]; 
\draw [fill, black] (2,2) circle [radius=.1]; 
\draw [fill, black] (2,4) circle [radius=.1]; 
\draw [fill, black] (3,3) circle [radius=.1]; 
\draw [fill, black] (3,5) circle [radius=.1]; 
\draw [fill, black] (4,4) circle [radius=.1]; 
\draw [fill, black] (4,6) circle [radius=.1]; 
\draw [fill, black] (5,5) circle [radius=.1]; 
\draw [fill, black] (6,6) circle [radius=.1]; 
\draw [dashed, thick] (0,0)--(6,6);
\node at (5.7,-0.5) {$\ell$};  \node at (-1,5.5) {de Rham};
\end{axis}
\end{tikzpicture}}
\]
\end{definition}

With this, we are finally ready to prove proposition \ref{PROP:Un-symmetrizing-Chern}.
\begin{proof}[Proof of proposition \ref{PROP:Un-symmetrizing-Chern}]
Note by the above definition \ref{DEF:ChL} $\Ch$ and $\ChL$ differs by multiples of the ${{\bf{b}}_{\sigma^{(k-1)}_{j(\sigma,k)}}}$s, which are concentrated in $u^\ell$ and $u^{\ell+1}$ powers. We only care about the ``diagonal'' terms in the $\ell$-de Rham diagram above. Note, that there is no symmetrization for $\ell=0$ and $\ell=1$.

Now, for $\ell\geq 1$, the ``diagonal'' terms coming from the ${\bf b}_\nu$s are given by \eqref{EQU:delta-on-b-ell} . We can use \eqref{EQU:Cech-diff-for-one-transpos} to simplify the right-hand side of \eqref{EQU:delta-on-b-ell} for $\nu=\sigma^{(k-1)}_{j}$ as follows:
\begin{multline*}
{\bf{b}}_{\sigma^{(k-1)}_{j}}(M)_n(\vec{\mc E})(e_{i_0,\dots, i_{\ell+1}})
=
u^{\ell+1}\cdot (-1)^j\cdot \Big(
\tr(\theta_{\sigma^{(k-1)}(1)}\dots \theta_{j}\dots \theta_{j+1}\dots \theta_{\sigma^{(k-1)}(\ell+1)})
\\
+
\tr(\theta_{\sigma^{(k-1)}(1)}\dots \theta_{j+1}\dots \theta_{j}\dots \theta_{\sigma^{(k-1)}(\ell+1)})
\Big),
\end{multline*}
where the second trace is $\tr(\theta_{\sigma^{(k)}(1)}\dots \theta_{\sigma^{(k)}(\ell+1)})$. Thus, since $\sgn(\sigma^{(k-1)})=-\sgn(\sigma^{(k)})$, any term $u^{\ell+1}\cdot \frac{\sgn(\sigma^{(k-1)})}{(\ell+1)!(2\pi i)^{\ell+1}}\cdot  \tr\big(\tilde{\theta}_{\sigma^{(k-1)}(1)}\dots \tilde{\theta}_{\sigma^{(k-1)}(\ell+1)}\big)$ gets replaced by a term $u^{\ell+1}\cdot \frac{\sgn(\sigma^{(k)})}{(\ell+1)!(2\pi i)^{\ell+1}}\cdot \tr\big(\tilde{\theta}_{\sigma^{(k)}(1)}\dots \tilde{\theta}_{\sigma^{(k)}(\ell+1)}\big)$. As $k=1, \dots, p(\sigma)$, this replaces each $\sigma^{(0)}=\sigma$ with $\sigma^{(p)}=(\ell+1,\ell,\dots, 2,1)$. Thus, the sum in \eqref{EQU:CH-U-0...l-lowest-ul} over all $\sigma\in S_{\ell+1}$ gets replaced with $(\ell+1)!$ copies of the term for the permutation $\sigma^{(p)}=(\ell+1,\dots,1)$.
\end{proof}

\appendix

\end{document}